\documentclass[a4paper,11pt,pdf]{amsart}

\usepackage{enumerate, amsmath, amsfonts, amssymb, amsthm, thmtools, wasysym, graphics, graphicx, xcolor, frcursive,comment,bbm}

\usepackage{etex}

\definecolor{darkblue}{rgb}{0.0,0,0.7} 
\newcommand{\darkblue}{\color{darkblue}} 
\definecolor{darkred}{rgb}{0.7,0,0} 
\newcommand{\emp}[1]{\emph{\darkblue #1}} 

\usepackage[all]{xy}
\usepackage[T1]{fontenc}

\usepackage[colorinlistoftodos]{todonotes}

\def\W{{\mathfrak{S}_{n+1}}}

\def\Bni{\mathcal{B}_{n+1}}
\def\eg{\mbox{\it eg}.}
\def\Pc{\mathcal{P}_c}

\def\Pol{{\sf{Pol}}}

\def\Pol{{\sf{Pol}}}
\def\Vert{{\sf{Vert}}}
\def\Wf{\mathcal{W}_f}

\def\bs{{\mathbf{s}}}

\def\ls{{\ell_{\mathcal{S}}}}

\newtheorem{thm}{Theorem}[section]
\newtheorem{prop}[thm]{Proposition}
\newtheorem{lem}[thm]{Lemma}
\newtheorem{cor}[thm]{Corollary}

\theoremstyle{remark}
\newtheorem*{nota}{Notation}

\theoremstyle{definition}
\newtheorem{defn}[thm]{Definition}
\newtheorem{rmq}[thm]{Remark}
\newtheorem{exple}[thm]{Example}

\usepackage{graphicx}                  
\usepackage{pstricks,pst-plot,pst-text,pst-tree,pst-eps,pst-fill,pst-node,pst-math}
\usepackage{setspace}

\newcommand{\ts}{\textsuperscript}

\title[Bases of Temperley-Lieb algebras]{Noncrossing partitions, fully commutative elements and bases of the Temperley-Lieb algebra}

\author{Thomas Gobet}
\address{TU Kaiserslautern, Fachbereich Mathematik, Postfach 3049, 67653 Kaiserslautern, Germany.}
\email{gobet@mathematik.uni-kl.de} 

\begin{document}
\maketitle
\begin{abstract}
We introduce a new basis of the Temperley-Lieb algebra. It is defined using a bijection between noncrossing partitions and fully commutative elements together with a basis introduced by Zinno, which is obtained by mapping the simple elements of the Birman-Ko-Lee braid monoid to the Temperley-Lieb algebra. The combinatorics of the new basis involve the Bruhat order restricted to noncrossing partitions. As an application we can derive properties of the coefficients of the base change matrix between Zinno's basis and the well-known diagram or Kazhdan-Lusztig basis of the Temperley-Lieb algebra. In particular, we give closed formulas for some of the coefficients of the expansion of an element of the diagram basis in the Zinno basis. 
\end{abstract}
\tableofcontents
\section{Introduction}
The Temperley-Lieb algebra $\mathrm{TL}_n$ (of type $A_n$) is an associative, unital $\mathbb{Z}[v,v^{-1}]$-algebra of dimension equal to the $(n+1)$\ts{th} Catalan number $C_{n+1}=\frac{1}{n+2}\binom{2(n+1)}{n+1}$. It is generated by $b_1,\dots, b_n$, with relations 
\begin{eqnarray*}
b_j b_i b_j&=&b_j~\text{if $|i-j|=1$},\\
b_i b_j&=&b_j b_i~\text{if $|i-j|>1$},\\
b_i^2&=&(v+v^{-1}) b_i.
\end{eqnarray*}
Alternatively, it can be viewed as a quotient algebra of the Iwahori-Hecke algebra $\mathcal{H}$ of type $A_n$. There is well-known diagrammatic version of $\mathrm{TL}_n$ which is due to Kauffman (see \cite{Kau}), where the $b_i$'s are represented by planar diagrams and multiplication is given by concatenation of diagrams. The corresponding diagram basis is indexed by fully commutative elements of the symmetric group. It is a monomial basis in the generators $b_1,\dots, b_n$, which is also the projection of the canonical Kazhdan-Lusztig basis for $\mathcal{H}$ (see \cite{KL}). The Kazhdan-Lusztig theory in $\mathrm{TL}_n$ is very simple, as reflected by the diagrammatic properties: for example, any product of the generators is proportional to an element of the basis. 

Other bases of $\mathrm{TL}_n$ are known. A particularly mysterious one is a basis introduced by Zinno in \cite{Z}. There is a multiplicative homomorphism from the braid group $\mathcal{B}_{n+1}$ on $n+1$ strands to the Temperley-Lieb algebra. The Zinno basis is obtained by mapping the so-called \emp{canonical factors} of the braid group to $\mathrm{TL}_n$ (via $\mathcal{H}$). The canonical factors are a set of distinguished elements of the Birman-Ko-Lee braid monoid (see \cite{BKL}), later generalized to the dual braid monoid by Bessis (see \cite{Dual}). The Birman-Ko-Lee or dual braid monoid embeds into the braid group, but is generated by a copy of the set of all the transpositions. The dual braid monoids are examples of \emp{Garside monoids} (see \cite{DP}, \cite{DDGKM}) and the more standard name for the canonical factors in that setting is the \emp{simple} elements or \emp{simples}. They can be seen as lifts of noncrossing partitions (viewed as elements of the symmetric group) to the braid group. The basis defined by Zinno is therefore naturally indexed by noncrossing partitions of $\mathfrak{S}_{n+1}$, which is another set enumerated by the Catalan number $C(n+1)$.

Zinno shows that the images of the simple elements in $\mathrm{TL}_n$ form a $\mathbb{Z}[v,v^{-1}]$-linear basis of it by defining a bijection between noncrossing partitions and fully commutative elements as well as a partial order on the set of simple elements. He then shows that there exists a matrix with respect to any linear extension of this partial order which is upper triangular with invertible coefficients on the diagonal, allowing one to pass from the diagram basis to the set of images of simple elements. Zinno's bijection is then read on the diagonal of the matrix. The bijection is given by an algorithm which extracts a subword of a specific braid word chosen to represent each simple element. The obtained subword is then (after surjection to the symmetric group) shown to be fully commutative. The approach is indirect and there is no description of the inverse bijection. 

In this paper, we reformulate Zinno's bijection in a simple way, allowing one to explicitly compute the inverse bijection. We then use this bijection to introduce a new basis of the Temperley-Lieb algebra. This basis will allow us to control a part of the base change matrix between the Zinno and diagram bases and find closed formulas for some of the coefficients of the matrix. Surprisingly, the new basis involves considering the Bruhat order on $\mathfrak{S}_{n+1}$ restricted to noncrossing partitions. Such an order is in fact stronger than the order defined by Zinno on noncrossing partitions to achieve triangularity. As a consequence, the new basis is an intermediate basis between the diagram and Zinno basis, and both base change matrices between them and the new basis are upper triangular with invertible coefficients on the diagonals if one orders the set of noncrossing partitions by any linear extension of the Bruhat order.\\    
~\\
\textbf{Acknowledgments}~The author thanks François Digne for reading preliminary versions of the paper and the referee for his careful reading of the manuscript and many helpful remarks and comments. 
  
\section{Bijections between fully commutative elements and noncrossing partitions}
\subsection{Fully commutative elements}\label{sec:fully}
Let $(\mathcal{W}, \mathcal{S})$ be a Coxeter system with length function $\ell_{\mathcal{S}}:\mathcal{W}\rightarrow \mathbb{Z}_{\geq 0}$.
\begin{defn}
An element $w\in\mathcal{W}$ is \emp{fully commutative} if one can pass from any reduced $\mathcal{S}$-decomposition of $w$ to any other by applying a sequence of relations of the form $st=ts$, where $s,t\in\mathcal{S}$.
\end{defn}
We denote by $\mathcal{W}_f$ the set of fully commutative elements of $\mathcal{W}$. For more on fully commutative elements see \cite{stem}. In this paper, the Coxeter systems considered are of type $A_n$. In that case there are many well-known equivalent characterizations of fully commutative elements. We list those we shall need in this paper below:
\begin{prop}\label{prop:fullya}
Let $(\mathcal{W},\mathcal{S})$ be of type $A_n$, with $\mathcal{W}\cong\mathfrak{S}_{n+1}$ and $\mathcal{S}:=\{s_i=(i,i+1)\}_{i=1}^n$. Let $w\in\mathcal{W}$. The following are equivalent:
\begin{enumerate}
\item The element $w$ is fully commutative,
\item If $s_{i_1}\cdots s_{i_k}$ is a reduced $\mathcal{S}$-decomposition of $w$, then for all $i=1,\dots,n$, the integer $n_i(w):=|\{j~|~i_j=i\}|$ is independent of the chosen reduced $\mathcal{S}$-decomposition,
\item The element $w$ has a reduced $\mathcal{S}$-decomposition of the form $$(s_{{i_{1}}} s_{{i_{1}-1}}\cdots s_{{j_{1}}})(s_{{i_{2}}} s_{{i_{2}-1}}\cdots s_{{j_{2}}})\cdots (s_{{i_\ell}} s_{{i_\ell-1}}\cdots s_{{j_\ell}})$$ where all the indices lie in $\{1, \dots, n\}$, $i_{1}<i_{2}<\dots <i_\ell$, $j_{1}<j_{2}<\dots <j_\ell$ and $j_m\leq i_m$ for all $m=1,\dots,\ell$,
\item If $s_{i_1}\cdots s_{i_k}$ is a reduced $\mathcal{S}$-decomposition of $w$ with $s_{i_j}=s_i=s_{i_d}$, $j<d$ and $s_{i_m}\neq s_i$ for all $j<m<d$, then $(s_{i_{j+1}},s_{i_{j+2}},\dots, s_{i_{d-1}})$ has exactly one entry equal to $s_{i+1}$ and exactly one entry equal to $s_{i-1}$.
\end{enumerate}
\end{prop}
\begin{proof}
The equivalence $1\Leftrightarrow 2$ is clear and true for any Coxeter system such that the only even entry of the Coxeter matrix is $2$. The last two conditions are often considered in the case of the so-called reduced words of the Temperley-Lieb algebra but the results still hold in the symmetric group and the proofs can easily be adapted: for the equivalence $1\Leftrightarrow 3$, see \cite[Section 2.8]{Goo}. The existence of canonical reduced $\mathcal{S}$-decompositions as in $3$ have been noticed by Jones in the Temperley-Lieb case in \cite[Section 3.5]{JonesSub} (see also \cite[Section 5.7]{KT}). For the equivalence $1\Leftrightarrow 4$, see \cite[Theorem 1]{Z}. 
\end{proof}
\begin{nota}
For $w\in\mathcal{W}_f$, we denote by $J_w$ the set $\{j_1,\dots, j_\ell\}$ from point $(3)$ of Proposition \ref{prop:fullya} and by $I_w$ the set $\{i_1,\dots, i_\ell\}$.
\end{nota}
Given a Coxeter system $(\mathcal{W}, \mathcal{S})$, recall that a simple transposition occurs in a reduced $\mathcal{S}$-decomposition of an element $w\in \mathcal{W}$ if and only if it occurs in any reduced $\mathcal{S}$-decomposition of $\mathcal{W}$. 

\begin{cor}\label{cor:caractij}
Let $(\mathcal{W}, \mathcal{S})$ be of type $A_n$. Let $w\in\mathcal{W}_f$, $i\in\{1,\dots, n\}$ such that $s_i$ occurs in any reduced $\mathcal{S}$-decomposition of $w$. Then $i\in I_w$ if and only if in any reduced $\mathcal{S}$-decomposition of $w$, there is no occurrence of $s_{i+1}$ before the first occurrence of $s_i$. Similarly $i\in J_w$ if and only if in any reduced $\mathcal{S}$-decomposition of $w$, there is no occurrence of $s_{i-1}$ after the last occurrence of $s_i$. 
\end{cor}
\begin{proof}
If $i\in I_w$, then the claimed property holds in the canonical reduced $\mathcal{S}$-decomposition of point $(3)$ of Proposition \ref{prop:fullya}. Since one passes from any reduced expression to any other only by applying commutation relations and $s_{i+1}$ does not commute with $s_i$, an $s_{i+1}$ in a reduced expression therefore cannot be moved to the left of the first $s_i$ in a reduced decomposition. 

Conversely, if there is never an occurrence of $s_{i+1}$ before the first occurrence of $s_i$ in a reduced expression, in particular it holds for the canonical reduced decomposition of point $(3)$ of Proposition \ref{prop:fullya}, and this is possible only if $i\in I_w$. The proof of the second statement is similar. 
\end{proof}

\subsection{Noncrossing partitions and dual braid monoid}\label{sec:noncdual}
From now and unless otherwise specified, $(\mathcal{W},\mathcal{S})$ will be of type $A_n$, with the notations introduced in Proposition \ref{prop:fullya}. The \emp{support} of a permutation $w\in\W$ is the set of $i\in\{1,\dots,n+1\}$ such that $w(i)\neq i$. Let us point out that the material presented below can be generalized to arbitrary finite Coxeter systems, see \cite{Dual}. 

Let $\mathcal{T}$ be the set of transpositions of $\mathcal{W}$ and $\ell_{\mathcal{T}}:\mathcal{W}\rightarrow \mathbb{Z}_{\geq 0}$ be the transposition length. There is a partial order $<_{\mathcal{T}}$ on $\mathcal{W}$ defined by $u<_{\mathcal{T}} v$ if the equality 
$$\ell_{\mathcal{T}}(u)+\ell_{\mathcal{T}}(u^{-1} v)=\ell_{\mathcal{T}}(v)$$
is satisfied. Let $c$ be any standard Coxeter element, that is, any product of all the $s_i$ in some order. Such an element is an $(n+1)$-cycle, but any $(n+1)$-cycle is not a standard Coxeter element: see \cite[Section 7]{GobWil} for a characterization of those $(n+1)$-cycles which are standard Coxeter elements. One has $\mathcal{T}\subset\mathcal{P}_c$ (see \cite[Lemma 1.2.1]{Dual}). The restriction of $<_{\mathcal{T}}$ to $\mathcal{P}_c:=\{x\in\mathcal{W}~|~x<_{\mathcal{T}}c\}$ endows $\mathcal{P}_c$ with a lattice structure. The obtained lattice is isomorphic to the lattice of noncrossing partitions (for the "is finer than" order) as shown in \cite{BDM}. In case $c=s_1\cdots s_n=(1,2,\dots,n+1)$, one obtains the noncrossing partition corresponding to $x\in\mathcal{P}_c$ by looking at the decomposition of $x$ into a product of disjoint cycles (this approach provides a generalization of noncrossing partitions to arbitrary finite Coxeter groups; see \cite{kapi} and \cite{Dual}). 

It is well-known that that noncrossing partitions are enumerated by $C(n+1)=\frac{1}{n+2}\binom{2(n+1)}{n+1}$, the $(n+1)$\ts{st} Catalan number. Recall that there is a graphical representation of a noncrossing partition $u\in\mathcal{P}_c$ by a disjoint union of polygons having vertices in a set of $n+1$ points on a circle labelled with $1,2,\dots, n+1$ in clockwise order, as in Figure \ref{figure:nonc}. The support of each cycle occurring in a decomposition of $u$ as a product of disjoint cycles is mapped to the polygon with vertices the elements of the support (we will assume that an edge is a polygon). 

From now, we assume that $c=s_1s_2\cdots s_n$. We will identify a noncrossing partition with the corresponding permutation of $\mathcal{P}_c$.

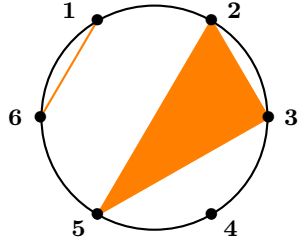
\begin{figure}[h!]
\begin{center}
\begin{tabular}{ccc}
& \begin{pspicture}(0,0)(6,3)
\pscircle(3,1.5){1.5}
\psdots(4.5,1.5)(3.75,2.79)(3.75,.21)(2.25,.21)(2.25,2.79)(1.5,1.5)
\pspolygon[linecolor=orange, fillstyle=solid, fillcolor=orange](3.75,2.79)(4.5,1.5)(2.25,.21)
\psline[linecolor=orange](2.25,2.79)(1.5,1.5)
\rput(3.75,.21){$\bullet$}
\rput(4.5,1.5){$\bullet$}
\rput(3.75,2.79){$\bullet$}
\rput(2.25,.21){$\bullet$}
\rput(2.25,2.79){$\bullet$}
\rput(1.5,1.5){$\bullet$}
\rput(4.05,2.92){\textrm{{\footnotesize \textbf{2}}}}
\rput(4.8,1.5){\textrm{{\footnotesize \textbf{3}}}}
\rput(4,.03){\textrm{{\footnotesize \textbf{4}}}}
\rput(2,.03){\textrm{{\footnotesize \textbf{5}}}}
\rput(1.17,1.5){\textrm{{\footnotesize \textbf{6}}}}
\rput(1.88,2.92){\textrm{{\footnotesize \textbf{1}}}}
\end{pspicture} & \\
\end{tabular}
\end{center}
\caption{Noncrossing partition corresponding to the permutation $x=(1,6)(2,3,5)\in\mathcal{P}_c$ for $c=s_1s_2s_3s_4s_5=(1,2,3,4,5,6)$ in type $A_5$.}
\label{figure:nonc}
\end{figure}

\begin{defn}[Bessis, \cite{Dual}] The \emp{dual braid monoid} associated to $(\mathcal{W},\mathcal{T},c)$ is generated by a copy $\{i_c(t)~|~t\in\mathcal{T}\}$ of $\mathcal{T}$ with relations $$i_c(t)i_c(t')=i_c(t')i_c(t'tt')~\text{if } tt'\in\mathcal{P}_c.$$ A relation as above is called a \emp{dual braid relation}. 
\end{defn}

The dual braid monoid is a generalization of the Birman-Ko-Lee monoid from \cite{BKL} which corresponds to a fixed choice of Coxeter element (see also \cite{BDM}). Bessis' definition works for arbitrary finite Coxeter systems. The terminology comes from the fact that there is an embedding $\iota_c:B_c^*\hookrightarrow \mathrm{Frac}(B_c^*)\cong \Bni$, where $\Bni$ is the braid group on $n+1$ strands. Recall that $\Bni$ is generated by a copy $\{\bs_i~|~s_i\in\mathcal{S}\}$ of the elements of $\mathcal{S}$ together with the relations
\begin{eqnarray*}
\bs_i\bs_{i+1}\bs_i&=&\bs_{i+1}\bs_i\bs_{i+1}, ~~\forall i\in\{1,\dots,n-1\}\\  \bs_i\bs_j&=&\bs_j\bs_i, ~~\text{if }|i-j|>1.
\end{eqnarray*} 
In the case where $c=s_1\cdots s_n$, the image $\iota_c(i_c(t))$ in $\Bni$ of $i_c(t)$ for $t=(i,k+1)$ with $k\geq i$ is represented by the braid word $$\bs_{i,k+1}:=\bs_k^{-1} \bs_{k-1}^{-1}\cdots \bs_{i+1}^{-1}\bs_i \bs_{i+1}\cdots \bs_k.$$

Moreover, the monoid $B_c^*$ shares many properties with the positive braid monoid $\Bni^+$ (that is, the monoid with the same generators and relations as $\Bni$); both turn out to be so-called \emp{Garside monoids} (see \cite{DP}), hence they embed into their group of fractions, which is in both cases isomorphic to $\Bni$ (for more on the general theory of Garside monoids we refer to \cite{DDGKM}). In particular as Garside monoid $B_c^*$ has a set of distinguished elements called the \emp{simple elements} or \emp{simples}. They are lifts of elements of $\mathcal{P}_c$ and can be defined combinatorially as follows:

\begin{defn} Let $x\in\mathcal{P}_c$ with a reduced $\mathcal{T}$-decomposition $t_1\cdots t_k$. We define the \emp{simple element} or \emp{simple} corresponding to $x$ and denoted by $i_c(x)$ as the product $$i_c(t_1)i_c(t_2)\cdots i_c(t_k).$$ Notice that this definition makes sense only if the product $i_c(t_1)i_c(t_2)\cdots i_c(t_k)$ is independent of the chosen reduced $\mathcal{T}$-decomposition. This holds for any $x\in\mathcal{P}_c$, as a consequence of the dual braid relations (see \cite[Section 1.6]{Dual}).
\end{defn}
 
\subsection{Bijections}\label{sec:bijn}
In this subsection, we introduce a bijection between the sets $\Wf$ and $\mathcal{P}_c$ from Subsections \ref{sec:fully} and \ref{sec:noncdual}. Let us begin with some notation:
\begin{nota}
Let $k\in\{1,\dots,n\}$. We denote by $\mathcal{I}_k$ the set of pairs $(X, Y)$ where $X=\{d_1, d_2,\dots, d_k\}$, $Y=\{e_1, e_2,\dots, e_k\}$, $e_i,d_i\in\{1,\dots,n+1\}$, $d_i<d_{i+1}$, $e_i<e_{i+1}$ for each $1\leq i<k$, $d_i< e_i$ for each $1\leq i\leq k$. Set $\mathcal{I}:=\coprod_{k=0}^n \mathcal{I}_k$. For $A\subset\mathbb{Z}$ and $n\in\mathbb{Z}$, we denote by $A[n]$ the set $\{a+n~|~a\in A\}$.
\end{nota}
Notice that the set $\mathcal{I}$ is item $107$ in \cite{Stanley}. In particular it is known that $|\mathcal{I}|=C(n+1)$. 
\begin{rmq}\label{rmq:wfi}
There is a bijection $\mathcal{W}_f\rightarrow \mathcal{I}$ given by $w\mapsto(J_w, I_w[1])$. This is just a reformulation of point $(3)$ of Proposition \ref{prop:fullya}.
\end{rmq}
\begin{nota}
Fix $x\in\mathcal{P}_c$ as in \ref{sec:noncdual}. We denote by $\Pol(x)$ the set of polygons of the graphical representation of $x$ described in Section \ref{sec:noncdual}.
\end{nota}

Let $P\in\Pol(x)$. Then $P$ is given by an ordered sequence of indices: $P=[i_1 i_2 \cdots i_k]$ where $i_1<i_2<\dots<i_k$ and $i_j$ are the vertices of $P$ (we identify the vertices with their labels). In the example of Figure \ref{figure:nonc} one has two polygons $P_1=[235]$ and $P_2=[16]$. For $m\in\{1,\dots, n+1\}$, we abuse notation and write $m\in P$ if there exists $1\leq j\leq k$ such that $m=i_j$.

\begin{defn} We say that $\min P:=i_1$ is an \emp{initial} index and $\max P:=i_k$ a \emp{terminal} index of $P$ or $x$. The \emp{longest edge} of $P$ is the edge joining $i_1$ to $i_k$. A vertex $m\in \{2,\dots,n\}$ is \emp{nested} in $P$ if there exists $1<j\leq k$ such that $i_{j-1}<m<i_j$. We say that $Q\in \Pol(x)$ is \emp{nested} in $P$ if any $m\in Q$ is nested in $P$. Note that since $x$ is a noncrossing partition, if one vertex of $Q$ is nested in $P$, then $Q$ must be nested in $P$.
\end{defn}

Let $\Pol(x)=\{P_1,\dots, P_r\}$. Any polygon $P_i=[i_1 i_2 \cdots i_k]\in\Pol(x)$ represents an element $y_i\in \mathcal{P}_c$. As an element of the symmetric group $y_i$ is $(i_1, i_2, \dots, i_k)$ and one has that $y_1 y_2\cdots y_r$ is the decomposition of $x$ into a product of disjoint cycles. Let $j<j'$. A reduced $\mathcal{S}$-decomposition for the transposition $(j, j')$ is given by the word $$[j, j']:=s_{j'-1} s_{j'-2}\cdots s_{j+1} s_j s_{j+1}\cdots s_{j'-2} s_{j'-1}.$$ Consider the word representing $y_i$ and obtained by the concatenation of such $[j,j']$'s $$m_i:=[i_1,i_2]\star[i_2,i_3]\star\cdots\star [i_{k-1},i_k].$$

\begin{lem}\label{stdreduced}
With the above notations, the concatenation $m_1\star m_2\star\cdots\star m_r$ is a reduced $\mathcal{S}$-decomposition of $x$. In particular one has $\sum_{i=1}^r \ell_{\mathcal{S}}(y_i)=\ell_{\mathcal{S}}(x)$.
\end{lem}

\begin{proof}
The proof is by induction on $r=|\Pol(x)|$. Let $r=1$. Recall that the length $\ell_{\mathcal{S}}(\sigma)$ of a permutation $\sigma\in\mathfrak{S}_{n+1}$ is equal to the number of $i<j$ such that $\sigma(i)>\sigma(j)$ (see \cite[Proposition 1.5.2]{BjBr}). Hence the length of a transposition $(j,j')$ with $j<j'$ is equal to $2(j'-j)-1$. Since $i_1<i_2<\dots <i_k$ it follows that $\ell_{\mathcal{S}}(m_i)$ is equal to $\sum_{j=1}^{k-1} (2 (i_{j+1}-i_j)-1)$. But each word $[i_j,i_{j+1}] $, $j=1,\dots, k-1$ has exactly $2(i_{j+1}-i_j)-1$ letters, which concludes the proof in that case. 

Now assume that $r>1$. Consider a polygon $P_m=[i_1i_2\cdots i_k]\in\Pol(x)$ which has no other polygon nested in it. The corresponding cycle is $y_m=(i_1,i_2,\dots, i_k)$. Given any $P=[j_1j_2\cdots j_\ell]\in\Pol(x)$ with $P\neq P_m$, one has either $j_\ell < i_1$, or $i_k<j_1$, or there exists $p\in\{1,\dots, \ell-1\}$ such that $j_p<i_1$, $i_k< j_{p+1}$ (in this last case, $P_m$ is nested in $P$). 

Let $x'=y_1\cdots \hat{y_m}\cdots y_r$, where the hat denotes omission. Since graphically we just removed one polygon of $x$, we have $x'\in\mathcal{P}_c$. Let $i,j\in\{1,\dots, n+1\}$, $i<j$. We must find all such pairs of indices $i,j$ such that $x(i)>x(j)$. We treat different cases by comparing the values of $x$ and $x'$ on $i$ and $j$. Since $|\Pol(x')|=|\Pol(x)|-1$ we will then conclude by induction. 

If neither $i$ nor $j$ lies in $\{i_1,i_2,\dots,i_k\}$, then $x(i)>x(j)$ if and only if $x'(i)>x'(j)$ since in that case we have $x(i)=x'(i)$, $x(j)=x'(j)$. 

If $i\in\{i_1, i_2,\dots,i_k\}$ and $j>i_k$, then $x(j)=x'(j)$. Moreover, $x(j)$ stays outside the interval $\{i_1,i_1+1,\dots,i_k\}$ thanks to the noncrossing property while $x(i)$ and $x'(i)=i$ both stay in $\{i_1,i_2,\dots,i_k\}$. Hence once again one has that $x(i)>x(j)$ if and only if $x'(i)>x'(j)$, and similarly if $j\in\{i_1,i_2,\dots,i_k\}$ and $i<i_1$. 

Now if $i\in\{i_1,i_2,\dots,i_k\}$ and $j\in\{i_1, i_1+1,\dots, i_k\}$ (or vice-versa), then since no polygon is nested in $P_m$ we have that $x'(i)=i$ and $x'(j)=j$. We have $x(i)>x(j)$ if and only if $y_m(i)>y_m(j)$ since all indices in $\{i_1,i_1+1,\dots,i_k\}$ which are not vertices of $P_m$ are fixed by $x$ (because no polygon is nested in $P_m$). 

To summarize the pairs $i<j$ with $x(i)>x(j)$ are exactly those for which $x'(i)>x'(j)$ or $y_m(i)>y_m(j)$. Since the various cases treated above also show that these last two conditions cannot be realized simultaneously it follows that $$\ell_{\mathcal{S}}(x)=\ell_{\mathcal{S}}(x')+\ell_{\mathcal{S}}(y_m),$$
and by induction the claim follows. 

\end{proof}


\begin{nota} Write $\Vert(x)$ for the set of vertices of polygons of $x$ and set $$U_x:=\Vert(x)\setminus \{\mbox{initial vertices}\},~ D_x:=\Vert(x)\setminus \{\mbox{terminal vertices}\}.$$ Notice that $|D_x|=|U_x|$ and $(D_x, U_x)\in\mathcal{I}$.  
\end{nota}
\begin{exple}
In the example of Figure \ref{figure:nonc} we have $D_x=\{1, 2, 3\}$, $U_x=\{3,5,6\}$. The integer $4$ is nested in both $P_1=[235]$ and $P_2=[16]$. The integers $3$ and $5$ are nested in $P_2$ but not in $P_1$. The integer $6$ is not nested in any polygon of $x$.
\end{exple}
\begin{lem}\label{lem:ancrage1}
Let $(D, U)\in\mathcal{I}$ with $D\cap U=\emptyset$. There is a unique $x\in\mathcal{P}_c$ such that $(D_x, U_x)=(D, U)$.
\end{lem}
\begin{proof}
Since $D\cap U=\emptyset$, the $x$ we have to find must be represented by a disjoint union of edges. Its set of initial indices must be equal to $D$ while its set of terminal indices must be equal to $U$. The proof is by induction on $|D|=|U|$.

If $D=U=\emptyset$, then the noncrossing partition $e$ is the unique one such that $(D_e, U_e)=(\emptyset, \emptyset)$. Now if $|D|=|U|>0$, consider the biggest index $d_j$ in $D$. It has to be joined to a unique $u_m\in U$ with $u_m>d_j$. To respect the noncrossing property $u_m$ must be the first index in $U$ appearing after $d_j$ when going along the circle in clockwise order. Indeed, assume that it is not the first one. Then there exists $u\in U$ which is nested in the edge $(d_j,u_m)$. The line containing the points $d_j$ and $u_m$ defines two half-planes $H_1$, $H_2$ and the point $u$ lies in one of them, say $H_1$. But $u$ must be joined to an index $d\in D$, $d\neq d_j$. These points lie in $H_2$ since they are before $d_j$ in clockwise order because $d_j$ is the biggest index in $D$. As a consequence the two segments $(d_j,u_m)$ and $(d,u)$ cross and the noncrossing property fails, a contradiction. 

Now if we consider $D'=D\backslash d_j$, $U'=U\backslash u_m$ and order them as $D'=\{d_1', \dots, d_{j-1}'\}$, $U'=\{u_1',\dots,u_{j-1}'\}$, $d_i'<d_{i+1}'$, $u_i'<u_{i+1}'$, we still have that $d_i'<u_i'$ since we removed the biggest index $d_j$ from $D$, hence $(D', U')\in\mathcal{I}$. By induction, there exists a unique $x'\in\Pc$ such that $(D_{x'}, U_{x'})=(D', U')$. The graphical representation of $x$ is obtained by adding the edge $(d_j, u_m)$ in the graphical representation of $x'$. The segment $(d_j,u_m)$ does not cross the segments coming from $x'$ since they all lie in the half-plane $H_2$ defined in the paragraph above. Hence the described process shows existence and uniqueness. 
\end{proof}
\begin{prop}\label{prop:terminalinitial}
The map $\varepsilon:\mathcal{P}_c\rightarrow\mathcal{I}$, $x\mapsto (D_x,U_x)$ is a bijection.
\end{prop}
\begin{proof}
Let $(D,U)\in\mathcal{I}$. We need to show that there is a unique $x\in\mathcal{P}_c$ such that $(D,U)=(D_x,U_x)$. Set $I:=D\backslash(D\cap U)$ and $T:=U\backslash(D\cap U)$. Write $I=\{d_1, d_2,\dots, d_j\}$, $T=\{u_1,u_2,\dots, u_j\}$ where $d_i<d_{i+1}$, $u_i<u_{i+1}$. Notice that since $(D,U)\in\mathcal{I}$ it follows that $d_i<u_i$ if $1\leq i\leq j$. If there exists $x\in \mathcal{P}_c$ with $(D,U)=(D_x,U_x)$, then $I$ must be the set of initial indices and $T$ the set of terminal indices of $x$. In particular, $|I|=|T|=\Pol(x)$ and any of the $u_i$'s is joined to a unique $d_k$ such that $(u_i, d_k)$ is the longest edge of a polygon of $x$.

We now show by induction on $D\cap U$ that we can always find such an $x$ and that it is uniquely determined. The case where $D\cap U=\emptyset$ is given by Lemma \ref{lem:ancrage1}. Assume that $D\cap U\neq\emptyset$. Let $r\in D\cap U$ and consider the pair $(D\backslash r, U\backslash r)$. It lies in $\mathcal{I}$ again and one has $(D\backslash r)\cap(U\backslash r)=(D\cap U)\backslash r$. By induction, there exists a unique noncrossing partition $x'$ such that $(D_{x'}, U_{x'})=(D\backslash r, U\backslash r)$. We claim that $r$ is nested in at least one polygon of $x'$. Indeed, assume that $r$ is nested in no polygon of $x'$. Then for any $P\in\Pol(x')$, one has either $\mathrm{max} P<r$ or $\mathrm{min} P>r$. It follows that $|\{a\in D~|~a<r\}|=|\{a\in U~|~a<r\}|$. This implies that when writing $D=\{a_1 < a_2 < \dots < a_\ell\}$, $U=\{b_1 < b_2 < \dots < b_\ell\}$, there exists $1\leq j\leq\ell$ such that $r=a_j=d_j$, a contradiction to $(D,U)\in\mathcal{I}$. Hence the claim holds. 

We then enlarge the polygon of $x'$ which is the closest to $r$ among the ones in which $r$ is nested to obtain a noncrossing partition $x$ with $(D_x, U_x)=(D,U)$. It is the only polygon among the polygons of $x'$ in which $r$ is nested to which we can add the vertex $r$ and keep the noncrossing property. Since $x'$ is by induction the unique noncrossing partition such that $(D_{x'}, U_{x'})=(D\backslash r, U\backslash r)$, the uniqueness of $x$ follows.

\end{proof}
Notice that the proof above explains in particular how to obtain $x\in\Pc$ from the pair of sets $(D_x, U_x)$. By Proposition \ref{prop:terminalinitial} together with Remark \ref{rmq:wfi} we get:
\begin{thm}\label{thm:bijref}
There exists a bijection $\varphi:\mathcal{P}_c\rightarrow\mathcal{W}_f$ characterized by the equality
$$(J_{\varphi(x)}, I_{\varphi(x)})=(D_x, U_x[-1]),~\text{for all }x\in\mathcal{P}_c.$$
We therefore have a characterization of the inverse bijection $\psi:\mathcal{W}_f\rightarrow\mathcal{P}_c$ by the equality $$(D_{\psi(w)}, U_{\psi(w)})=(J_w, I_w[1]),~\text{for all }w\in\mathcal{W}_f.$$
\end{thm}
\begin{exple}\label{ex:gob}
For $x$ as in Figure \ref{figure:nonc} we have $(D_x, U_x)=(\{1,2,3\},\{3,5,6\})$. Hence $(J_{\varphi(x)}, I_{\varphi(x)})=(\{1,2,3\},\{2,4,5\})$. Writing $\varphi(x)$ as in point $(3)$ of Proposition \ref{prop:fullya} we therefore have $$\varphi(x)=(s_2 s_1)(s_4 s_3 s_2)(s_5 s_4 s_3).$$
\end{exple}
\section{Zinno basis of the Temperley-Lieb algebra}
The aim of this section is to introduce results by Zinno (see \cite{Z}) on the classical Temperley-Lieb algebra and explain the relation with the previously introduced bijections. We first introduce the Temperley-Lieb algebra and explain the link with the braid group. 
\subsection{Temperley-Lieb algebra and braid group} 

\begin{defn}\label{defn:templieb}
The \emp{Temperley-Lieb algebra} $\mathrm{TL}_n$ is the associative, unital $\mathbb{Z}[v, v^{-1}]$-algebra having as generators $b_1,\dots, b_n$ and relations
\begin{eqnarray*}
b_j b_i b_j&=&b_j~\text{if $|i-j|=1$},\\
b_i b_j&=&b_j b_i~\text{if $|i-j|>1$},\\
b_i^2&=&(v+v^{-1}) b_i.
\end{eqnarray*}
\end{defn}
\noindent It has a basis indexed by fully commutative elements:
  \begin{prop}[Jones, \cite{JonesSub}]\label{prop:jones} 
Let $w\in \mathcal{W}_f$. One associates to any reduced $\mathcal{S}$-decomposition $s_{i_1} s_{i_2}\cdots s_{i_k}$ of $w$ the element $b_{i_1} b_{i_2}\cdots b_{i_k}$ of $\mathrm{TL}_n$.
  \begin{enumerate}
  \item The product $b_{i_1} b_{i_2}\cdots b_{i_k}$ is independent of the choice of the reduced expression for $w$.
  \item The set $\{b_w\}_{w\in\mathcal{W}_f}$ is a $\mathbb{Z}[v,v^{-1}]$-basis of $\mathrm{TL}_n$.
  \item Given any sequence $j_1 j_2\cdots j_m$ of integers in $\{1,\dots,n\}$, there exists a unique pair $(x,k)\in\mathcal{W}_f\times\mathbb{Z}_{\geq 0}$ such that
  $$b_{j_1} b_{j_2}\cdots b_{j_m}=(v+v^{-1})^k b_x.$$
  \end{enumerate}
  \end{prop}
\begin{nota}
We denote by $b_w$ the element $b_{i_1} b_{i_2}\cdots b_{i_k}$ in point $(1)$ of Proposition \ref{prop:jones}.
\end{nota}
The basis $\{b_w\}_{w\in\mathcal{W}_f}$ has a well-known interpretation by planar diagrams (see \cite{Kau} and also \cite[Section 5.7.4]{KT} and the references therein). We write $\mathbb{Z}[v,v^{-1}] \Bni$ for the group algebra of $\Bni$ over $\mathbb{Z}[v,v^{-1}]$. Recall that we use bold notation $\bs_i$ for the braid group generator which is the lift of $s_i$. There are two quotient maps (see Remark \ref{rmq:two quotients} for references): 
\begin{eqnarray*}
\omega: \mathbb{Z}[v,v^{-1}] \Bni&\twoheadrightarrow& \mathrm{TL}_n\\
\bs_i\mapsto v^{-1}-b_i,\\
\omega': \mathbb{Z}[v,v^{-1}] \Bni&\twoheadrightarrow& \mathrm{TL}_n\\
\bs_i\mapsto b_i-v.
\end{eqnarray*}
\begin{rmq}\label{rmq:two quotients}
Both $\omega$ and $\omega'$ factor through the \emp{Iwahori-Hecke algebra} $\mathcal{H}$ of the symmetric group via the natural quotient map $\pi:\mathbb{Z}[v,v^{-1}] \Bni\twoheadrightarrow\mathcal{H}$, $\bs_i\mapsto v T_{s_i}$ where $T_{s_i}$ is the standard generator of $\mathcal{H}$ (see \eg \cite{KL} or \cite[Section 4.2.1]{KT}). There are then two ways to realize $\mathrm{TL}_n$ as a quotient of $\mathcal{H}$ by maps $\theta, \theta'$ defined by $\theta(T_{s_i})=v^{-2}-v^{-1} b_i$, $\theta'(T_{s_i})=v^{-1} b_i-1$ so that $\omega=\theta\circ\pi$, $\omega'=\theta'\circ\pi$ (see \eg \cite[Section 2.3 and Remark 2.4]{Ram}). From a representation theoretic point of view, it reflects the fact that in the semisimple case, the Temperley-Lieb quotient can be obtained from the Iwahori-Hecke algebra either by taking the quotient by the standard tableaux having more than two columns or by by taking the quotient by the standard tableaux having more than two rows (for the representation theoretic approach see \cite[Section 5]{KT}). The algebra $\mathcal{H}$ has two canonical Kazhdan-Lusztig bases $\{C_w\}_{w\in\mathcal{W}}$ and $\{C_w'\}_{w\in\mathcal{W}}$ as defined in \cite{KL}. One has the following projections of bases: $\theta(C_w)=(-1)^{\ell_{\mathcal{S}}(w)} b_w$ if $w\in\mathcal{W}_f$, while $\theta(C_w)=0$ if $w\notin\mathcal{W}_f$. Similarly, one has $\theta'(C_w')=b_w$ if $w\in\mathcal{W}_f$, while $\theta'(C_w')=0$ if $w\notin\mathcal{W}_f$ (see \cite{FG}).
\end{rmq}
\noindent The various facts mentioned above motivate the following definition:
\begin{defn} The basis $\{b_w\}_{w\in\mathcal{W}_f}$ is the \emp{diagram} or \emp{Kazhdan-Lusztig} basis of $\mathrm{TL}_n$.
\end{defn}

In this paper, we work with the quotient map $\omega$. That is, given a braid word, we will obtain its image in $\mathrm{TL}_n$ by replacing $\bs_i$ by $v^{-1}-b_i$ and $\bs_i^{-1}$ by $v-b_i$ (one has $(v^{-1}-b_i)(v-b_i)=1$ as a consequence of the last relation in Definition \ref{defn:templieb}). Of course, the results can be adapted if one prefers to use the quotient map $\omega'$.
\subsection{Zinno basis}\label{sub:zinnob}
In this section, we introduce terminology and work of Zinno \cite{Z} and then show that the bijection described in Theorem \ref{thm:bijref} and a bijection given by an algorithm in \cite{Z} are the same. 

The images $\iota_c(i_c(x))$, $x\in\mathcal{P}_c$ of the simple elements of $B_c^*$ in $\Bni$ introduced in Section \ref{sec:noncdual}, which can be considered as lifts of noncrossing partitions in the braid group, are called \emp{canonical factors} (shortly \emp{canfacs}) by Zinno; we will call them \emp{simples} even it they are viewed in $\Bni$ since this is a more standard name for them. In fact, the way Zinno writes the canfacs corresponds to the simple elements of $B_{c'}^*$ where $c'$ is the Coxeter element $c'=s_n\cdots s_2 s_1$. Since we are working with $c=s_1 s_2\cdots s_n$ we need to reverse the order of the braid words considered in \cite{Z}. 

Recall that for $t=(i,k+1)$, $k\geq i$, the image of $i_c(t)$ in $\Bni$ is given by the braid word $$\bs_{i,k+1}:=\bs_k^{-1} \bs_{k-1}^{-1}\cdots \bs_{i+1}^{-1}\bs_i \bs_{i+1}\cdots \bs_k.$$
We will abuse notation and also write $i_c(x)$ for the image of a simple element in the braid group which we previously denoted by $\iota_c(i_c(x))$ since it will make no possible confusion. A braid word such as $\bs_{i,k+1}$ is called a \emp{syllable} by Zinno. The braid group generator $\bs_i$ is the \emp{center} of the syllable, splitting the syllable into a \emp{left} part $\mathbf{s}_k^{-1} \mathbf{s}_{k-1}^{-1}\cdots \mathbf{s}_{i+1}^{-1}$ and a \emp{right} part $\mathbf{s}_{i+1}\cdots \mathbf{s}_k$. The letters $\mathbf{s}_k^{\pm 1}$ are at the \emp{top} of the syllable. A noncrossing partition $x\in\mathcal{P}_c$ which is a cycle, that is, which is represented by a single polygon is still called a \emp{cycle} in \cite{Z} after lifting in $\Bni$. Zinno uses the following braid word to represent $i_c(x)$: firstly write $x=(i_1,i_2,\dots, i_k)$, where $i_1<i_2<\dots<i_k$. We have $$x=(i_1, i_2)(i_2,i_3)\cdots (i_{k-1},i_k)$$ and $\ell_{\mathcal{T}}(x)=k-1$. Then the cycle $i_c(x)$ is represented by the braid word $$\mathbf{s}_{i_1,i_2} \mathbf{s}_{i_2,i_3}\cdots \mathbf{s}_{i_{k-1},i_k}.$$ Now if $x$ has possibly more than one cycle, we will represent $i_c(x)$ by the braid word obtained by concatenating the cycles, ordered by the maximal index in each cycle (that is, the terminal index of the associated polygon) in ascending order, and refer to such a word as to the \emp{standard form} of a simple element of the dual braid monoid. We denote the obtained braid word by $\mathbf{m}_x$. 
\begin{exple}
Let $x=(1,6)(2,3,5)$ as in Figure \ref{figure:nonc}. There are two polygons $P_1=[235]$ and $P_2=[16]$. They have as corresponding standard form $\mathbf{s}_{1,6}$ and $\mathbf{s}_{2,3}\mathbf{s}_{3,5}$. We have $P_1<P_2$ since the maximal index of $P_1$ is $5$ and that of $P_2$ is $6$. Hence
$$\mathbf{m}_x=\mathbf{s}_2 \mathbf{s}_4^{-1}\mathbf{s}_3\mathbf{s}_4\mathbf{s}_5^{-1}\mathbf{s}_4^{-1}\mathbf{s}_3^{-1}\mathbf{s}_2^{-1}\mathbf{s}_1\mathbf{s}_2\mathbf{s}_3\mathbf{s}_4\mathbf{s}_5.$$
\end{exple}
\begin{rmq}\label{rmq:centrebebe}
Notice that a braid group generator can be the center of at most one syllable, hence it occurs twice in any other syllable in which it occurs, once in the left part with negative exponent and once in the right part with positive exponent. The way the polygons (equivalently the cycles) are ordered implies that if $\mathbf{s}_i$ is the center of a syllable, then the first occurrence of $\mathbf{s}_i^{\pm 1}$ in $\mathbf{m}_x$ when reading the word from the left to the right is at the center of that syllable and hence with positive exponent. Note that by definition of $\mathbf{m}_x$, one has that $\mathbf{s}_i$ is the center of a syllable if and only if $i$ is a non terminal index of a polygon of $x$, that is, if and only if $i\in D_x$.
\end{rmq}
\begin{defn} If we replace each $\bs_i^{\pm 1}$ by $s_i$ in $\mathbf{m}_x$, then by Lemma \ref{stdreduced} we obtain a reduced $\mathcal{S}$-decomposition $m_x$ of $x\in\mathcal{P}_c$ which we also call the \emp{standard form} of $x$. We will also call $m_t$ for $t\in\mathcal{T}$ a \emp{syllable} with a \emp{center}, \emp{left part}, etc.
\end{defn}

A more general definition of this Coxeter word is given in \cite{GobWil} where we work with arbitrary standard Coxeter elements. It turns out that the Coxeter word $m_x$ plays an important role in the study of a basis of the Temperley-Lieb algebra discovered by Zinno (which we will introduce in a few lines) and its generalizations to arbitrary Coxeter elements. 

Let $x\in\mathcal{P}_c$. Set $Z_x:=\omega(i_c(x))$. Notice that $\mathcal{S}\subset\mathcal{P}_c$, whence $Z_s=\omega(\mathbf{s})=v^{-1}-b_i$ for any $s\in\mathcal{S}$.
\begin{thm}[{\cite[Theorem 2]{Z}}]\label{thm:z}
The set $\{Z_x\}_{x\in\mathcal{P}_c}$ is a $\mathbb{Z}[v,v^{-1}]$-linear basis of $\mathrm{TL}_n$. 
\end{thm}

See also \cite{Lee}, where an alternative proof of this result is given. We follow \cite{Z} here, since it gives information on the relation with the basis $\{b_w\}_{w\in\Wf}$ which we will explore further. 
 
Zinno proves that there are total orders on $\{Z_x\}_{x\in\mathcal{P}_c}$ and $\{b_w\}_{w\in\mathcal{W}_f}$ such that there exists an upper triangular matrix with an explicitly computed invertible coefficient on the diagonal allowing one to pass from $\{b_w\}$ to $\{Z_x\}$. Since $\{b_w\}$ is a basis it follows that $\{Z_x\}$ is also a basis. He proceeds as follows: given $x\in\mathcal{P}_c$, Zinno considers the braid word $\mathbf{m}_x$ representing $i_c(x)$. He then extracts a subword (by subword we mean substring) $\mathbf{w}_x$ of $\mathbf{m}_x$ according to the following rules

\begin{itemize}
\item If a syllable has at least one letter indexed by $i$ (the letters indexed by $i$ are $\bs_i$ and $\bs_i^{-1}$), then that syllable must contribute to the subword exactly one of its letters indexed by $i$. In particular each center contributes since it is the only letter with its index in a syllable. 

\item If $\bs_i$ is the center of a syllable and occurs in another syllable, then such a syllable contributes the $\bs_i^{\pm 1}$ which has positive exponent. If $\bs_i$ is not the center of any syllable but there are syllables containing letters indexed by $i$, then these syllables must contribute their $\bs_i^{-1}$ to the subword. 
\end{itemize}

In this way we extract a subword $\mathbf{w}_x$. By replacing the $\bs_i^{\pm 1}$ by $s_i$ we get a Coxeter word $w_x$ for a permutation. Zinno then shows that $w_x$ is a reduced expression of a fully commutative element (we will therefore often abuse notation and identify $w_x$ with the fully commutative element it represents). Hence the map
$$a:\mathcal{P}_c\rightarrow \mathcal{W}_f,~x\mapsto w_x$$ is well-defined and Zinno shows that it is surjective. Since $|\mathcal{P}_c|=|\mathcal{W}_f|$ it follows that $a$ is bijective. 

\begin{rmq}
Thanks to remark \ref{rmq:centrebebe}, the rules given above are equivalent to the rules given by the following algorithm: read the word $\mathbf{m}_x$ from left to right. If the first letter $\bs_i^{\pm 1}$ occuring in $\mathbf{m}_x$ has positive (resp. negative) exponent, then all the occurrences of $\bs_i$ (resp. of $\bs_i^{- 1}$) in $\mathbf{m}_x$ and only those must contribute to the subword $\mathbf{w}_x$. Apply the same process to the next generator $\bs_j^{\pm 1}$, $j\neq i$ occuring to the right of the first $\bs_i^{\pm 1}$ in $\mathbf{m}_x$, until you have considered all the indices $k$ such that $\bs_k^{\pm 1}$ occurs in $\mathbf{m}_x$.
\end{rmq} 

An example of Zinno's algorithm to extract the fully commutative element $w_x$ as a subword of a standard form $\mathbf{m}_x$ of $i_c(x)$ for $x$ as in Figure \ref{figure:nonc} is given now.

\begin{exple}[Zinno's algorithm to extract $w_x=a(x)$ as a subword of $\mathbf{m}_x$]\label{ex:zinn}
~\\
Let $x=(2,3,5)(1,6)\in\mathcal{P}_{c}$.\\

~~~~~$\mathbf{m}_x=\bs_2(\bs_4^{-1} \bs_3 \bs_4)(\bs_5^{-1}\bs_4^{-1}\bs_3^{-1}\bs_2^{-1}\bs_1\bs_2\bs_3\bs_4\bs_5)$

~~~~~$\mathbf{m}_x={\underline{\bs_2}}(\bs_4^{-1} \bs_3 \bs_4)(\bs_5^{-1}\bs_4^{-1}\bs_3^{-1}\bs_2^{-1}\bs_1{\underline{\bs_2}}\bs_3\bs_4\bs_5)$

~~~~~$\mathbf{m}_x={\underline{\bs_2}}({\underline{\bs_4}^{-1}} \bs_3 \bs_4)(\bs_5^{-1}{\underline{\bs_4}^{-1}}\bs_3^{-1}\bs_2^{-1}\bs_1{\underline{\bs_2}}\bs_3\bs_4\bs_5)$

~~~~~$\mathbf{m}_x={\underline{\bs_2}}({\underline{\bs_4}^{-1}} {\underline{\bs_3}} \bs_4)(\bs_5^{-1}{\underline{\bs_4}^{-1}}\bs_3^{-1}\bs_2^{-1}\bs_1{\underline{\bs_2}}{\underline{\bs_3}}\bs_4\bs_5)$

~~~~~$\mathbf{m}_x={\underline{\bs_2}}({\underline{\bs_4}^{-1}} {\underline{\bs_3}} \bs_4)({\underline{\bs_5}^{-1}}{\underline{\bs_4}^{-1}}\bs_3^{-1}\bs_2^{-1}\bs_1{\underline{\bs_2}}{\underline{\bs_3}}\bs_4\bs_5)$

~~~~~$\mathbf{m}_x={\underline{\bs_2}}({\underline{\bs_4}^{-1}} {\underline{\bs_3}} \bs_4)({\underline{\bs_5}^{-1}}{\underline{\bs_4}^{-1}}\bs_3^{-1}\bs_2^{-1}{\underline{\bs_1}}{\underline{\bs_2}}{\underline{\bs_3}}\bs_4\bs_5)$\\

~~~~~$\rightsquigarrow \mathbf{w}_x={{\bs_2}}{{\bs_4}^{-1}}{{\bs_3}}{{\bs_5}^{-1}}{{\bs_4}^{-1}}{{\bs_1}}{{\bs_2}}{{\bs_3}}$\\

~~~~~$\rightsquigarrow w_x=s_2 s_4 s_3s_5s_4s_1s_2s_3={(s_2 s_1)(s_4 s_3 s_2)(s_5 s_4 s_3)}\in\mathcal{W}_f$.
\end{exple}
Notice that for $x=(1,3,5)(1,6)$, we have thanks to Examples \ref{ex:zinn} and \ref{ex:gob} that $a(x)=w_x=\varphi(x)$. This is a general fact: 
\begin{prop}\label{prop:lesmemes}
The bijection $\varphi:\mathcal{P}_c\rightarrow \mathcal{W}_f$ of Theorem \ref{thm:bijref} and the bijection described in \cite[Theorems 3 and 6]{Z} which we denoted by $a$ are the same.
\end{prop}
\begin{proof}
For $w\in\mathcal{W}_f$, we use the characterization of the sets $I_w$ and $J_w$ given in Corollary \ref{cor:caractij}.

Let $x\in \mathcal{P}_c$. By Theorem \ref{thm:bijref}, it suffices to show that $I_{w_x}[1]=U_x$ and $J_{w_x}= D_x$. We only show the first equality, the proof of the second one is similar. 

Let $i\in I_{w_x}$. By Corollary \ref{cor:caractij} there is no occurrence of $s_{i+1}$ before the first occurrence of $s_i$ in $w_x$. We claim that the first occurrence of $s_i$ in $w_x$ must come from a syllable $w$ of $\mathbf{w}_x$ whose first letter is $\bs_i^{\pm 1}$. Indeed, otherwise $\mathbf{s}_{i+1}^{-1}$ would occur in $w$ on the left of the $\mathbf{s}_i^{\pm 1}$ contributed and that $\mathbf{s}_{i+1}^{-1}$ would be contributed to $\mathbf{w}_x$ in case $\mathbf{s}_{i+1}$ is not a center. In case $\mathbf{s}_{i+1}$ is a center, the occurrence of $\mathbf{s}_{i+1}$ at the center must be the first in the word (by Remark \ref{rmq:centrebebe}), before $w$, and must be contributed. Hence the claim holds. But $\mathbf{s}_i^{\pm 1}$ is the first letter of a syllable if and only if $\mathbf{s}_i^{\pm 1}$ is at the top of that syllable, which holds if and only if $i+1\in U_x$. Hence $I_{w_x}[1]\subset U_x$. 

Conversely, let $i+1\in U_x$. Then $i+1$ is a vertex of a polygon $P\in\Pol(x)$ which is not initial. Write $(i_1, \dots, i_m)$, $i_1<i_2<\dots<i_m$ for the cycle corresponding to $P$. Let $i+1=i_\ell$, $1<\ell\leq m$. Then $\mathbf{s}_i^{\pm 1}$ is the first letter of the syllable $w=\mathbf{s}_{i_{\ell-1},i_\ell}$ of the cycle corresponding to $P$. We will show that this letter contributes to $\mathbf{w}_x$, that it is the first occurrence of $\mathbf{s}_i^{\pm 1}$ in $\mathbf{m}_x$ and that there is no occurrence of $\mathbf{s}_{i+1}^{\pm 1}$ in $\mathbf{m}_x$ at its left. These properties together imply that $i\in I_{w_x}$ by Corollary \ref{cor:caractij}. If there is another letter $\mathbf{s}_i^{\pm 1}$ before $w$, then it must be in a cycle corresponding to a polygon $Q\neq P$. Suppose that it occurs as a center of a syllable of the cycle corresponding to $Q$. This means that there exists $Q\in\Pol(x)$ with the vertex $i$ which is not terminal and $P\in\Pol(x)$ with the vertex $i+1$ which is not initial, contradicting the noncrossing property. If it is not a center, it cannot be at a top since $\mathbf{s}_i^{\pm 1}$ is already at the top of $w$ and there can be at most one syllable having it at its top. This implies that it has to be a letter of a syllable $\mathbf{s}_{k, k'}$ where $k<\min P$, $k'>\max P$ since the polygons are disjoint and noncrossing. If $Q$ is the polygon whose cycle has as its syllable $\mathbf{s}_{k,k'}$, we would have $\max Q>\max P$, hence $\mathbf{s}_{k, k'}$ would occur after $w$ in $\mathbf{m}_x$. Hence our $\mathbf{s}_i^{\pm 1}$ from $w$ is the first occurrence of $\mathbf{s}_i^{\pm 1}$ in $\mathbf{m}_x$. Now suppose $\mathbf{s}_{i+1}$ occurs in a syllable. If it is the center then it is in $P$ and the corresponding syllable appears just after $w$. If it is not the center, then to respect the noncrossing property one must again have that the syllable containing it appears after $w$. Therefore we have $i\in I_{w_x}$, hence $U_x\subset I_{w_x}[1]$.

\end{proof}

\section{A new basis of the Temperley-Lieb algebra}

Let us first recall some basic facts about the Bruhat order on a Coxeter group.
\subsection{Bruhat order}
We recall the definition and various characterizations of the Bruhat order on a Coxeter system $(\mathcal{W},\mathcal{S})$. For $w, w'\in\mathcal{W}$, we define a relation by $w\rightarrow w'$ if there exists $t\in\mathcal{T}$ such that $w'=tw$ and $\ell_{\mathcal{S}}(w)<\ell_{\mathcal{S}}(w')$. We then extend this relation to a partial order $<$ by setting $w< w'$ if there exist $w_1,\dots, w_k\in\mathcal{W}$ such that $$w\rightarrow w_1\rightarrow w_2\rightarrow\dots \rightarrow w_k\rightarrow w'.$$ It is the \emp{Bruhat order} of the Coxeter system $(\mathcal{W},\mathcal{S})$. The following characterization is classical (see \eg \cite[Corollary 2.2.3]{BjBr}). Recall that by \emp{subword} we mean substring: 
\begin{prop} 
For $w, w'\in\mathcal{W}$, the following are equivalent:
\begin{enumerate}
\item One has $w< w'$,
\item Any $\mathcal{S}$-reduced expression for $w'$ has a subword that is an $\mathcal{S}$-reduced expression for $w$,
\item There exists an $\mathcal{S}$-reduced expression for $w'$ which has a subword that is an $\mathcal{S}$-reduced expression for $w$.
\end{enumerate}
\end{prop}
\subsection{A new basis of the Temperley-Lieb algebra}
Let $(\mathcal{W},\mathcal{S})$ be of type $A_n$. Recall that we are working exclusively with the Coxeter element $c=s_1s_2\cdots s_n$. For $w\in\mathcal{W}_f$ we set $$L(w)=\{s\in \mathcal{S}~|~sw< w\},$$ $$R(w)=\{s\in \mathcal{S}~|~ws< w\}.$$

Note that these are just the left and right descent sets of $w$ viewed as permutation. 

\begin{rmq}\label{descents} Notice that if $s,t\in L(w)$, then $ts=st$. Indeed, if $s_i$, $s_{i+1}\in L(w)$, then $w$ has a reduced expression beginning with $s_i$ and another one beginning with $s_{i+1}$, hence cannot be fully commutative since one cannot move an $s_i$ after an $s_{i+1}$ only with commutation relations. The same holds if both $s,t$ lie in $R(w)$. Moreover, if $s\in L(w)$ (resp. $R(w)$), then $sw$ (resp. $ws$) lies again in $\mathcal{W}_f$ (see \cite[Proposition 2.4]{stem}). It follows together with point $(3)$ of Proposition \ref{prop:fullya} and Corollary \ref{cor:caractij} that $s_i\in L(w)$ if and only if $i\in I_w$ and $i-1\notin I_w$. 
\end{rmq}

Given a transposition $(i,j)\in\mathcal{T}$, assume that $P\in\Pol(x)$ has an edge joining point $i$ to point $j$. We will also denote this edge by $(i,j)$. Notice that if $(i,j)$ is an edge or a diagonal of a polygon of $x$, then $(i,j)<_{\mathcal{T}} x$ (see \cite[Proposition 2.6]{Brady}).

\begin{prop}\label{prop:configurations}
Let $w\in \mathcal{W}_f$, $s_i\in \mathcal{S}$. Then 
\begin{enumerate}
	\item $s_i\in L(w)$ if and only if $\{(i, i+1)$ is an edge of a polygon $P\in\Pol(\psi(w))$ with $i$ initial$\}$ or $\{$the point with index $i$ is not a vertex of a polygon of $\psi(w)$ but there exists a polygon $P\in\Pol(\psi(w))$ having an edge $(k, i+1)$ for some $k<i\}$.
	\item $s_i\in R(w)$ if and only if $\{(i, i+1)$ is an edge of a polygon $P\in\Pol(\psi(w))$ with $i+1$ terminal$\}$ or $\{$the point with index $i+1$ is not a vertex of a polygon of $\psi(w)$ but there exists a polygon $P\in\Pol(\psi(w))$ having an edge $(i, i+k)$ with $k>1\}$. 
\end{enumerate}
\end{prop}
\begin{proof}
One has that $s_i\in L(w)$ if and only if $i\in I_w$, $i-1\notin I_w$ (see Remark \ref{descents}). This holds if and only if $i+1\in U_{\psi(w)}$, $i\notin U_{\psi(w)}$ if and only if $\{(i, i+1)$ is an edge of a polygon of $\psi(w)$ with $i$ initial$\}$ or $\{i$ is not a vertex of a polygon of $\psi(w)$ but there exists an edge $(k, i+1)$ of a polygon with $k<i\}$. One argues similarly for $s_i\in R(w)$. 
\end{proof}
\begin{cor}\label{cor:bruhat}
Let $w\in\mathcal{W}_f$.
\begin{enumerate}
	\item If $s\in L(w)$, then $s\psi(w)\in \mathcal{P}_c$ and $\ell_{\mathcal{S}}(s\psi(w))=\ell_{\mathcal{S}}(\psi(w))-1$.
	\item If $s\in R(w)$, then $\psi(w)s\in \mathcal{P}_c$ and $\ell_{\mathcal{S}}(\psi(w)s)=\ell_{\mathcal{S}}(\psi(w))-1$.
\end{enumerate}
\end{cor}
\begin{proof}
By Proposition \ref{prop:configurations} we know what the assumption $s\in L(w)$ means in terms of the geometrical representation of $\psi(w)$ by a disjoint union of polygons. In case $(i,i+1)$ is an edge of a polygon $P$ of $\psi(w)$ with $i$ initial, it means that the cycle $y\in\mathcal{P}_c$ corresponding to $P$  is equal to $s_i y'$ where $y'\in\mathcal{P}_c$ is the cycle corresponding to the polygon $P$' obtained from $P$ by removing the vertex with index $i$. Since $i$ is the miminal index of $P$ one then has that $\ell_{\mathcal{S}}(y')=\ell_{\mathcal{S}}(y)-1$. But if a noncrossing partition $x\in\mathcal{P}_c$ has decomposition into disjoint cycles $y_1 y_2\cdots y_k$, one has by Lemma \ref{stdreduced} that $$\ell_{\mathcal{S}}(y)=\sum_{j=1}^k \ell_{\mathcal{S}}(y_j),$$
which concludes the proof of this case. In case $i$ is not an index of a vertex of a polygon of $\psi(w)$ but there is a polygon $P$ having an edge $(k, i+1)$ for $k<i$, consider again the cycle $y\in\mathcal{P}_c$ corresponding to $P$. The product $s_i y$ is again a noncrossing partitions corresponding to the polygon $P'$ obtained from $P$ by adding the vertex labelled by $i$. If the set of indices of vertices of $P$ is given by ${d_1,\dots, d_k}$, $d_j<d_{j+1}$ with $d_m=k, d_{m+1}=i+1$, an $\mathcal{S}$-reduced expression of $y$ is given by the concatenation $$[d_1,d_2]\star[d_2,d_3]\star\cdots\star [d_{k-1},d_k],$$
where $[j, \ell]=s_{\ell-1} s_{\ell-2}\cdots s_{j+1} s_j s_{j+1}\cdots s_{\ell-2} s_{\ell-1}$ (see Section \ref{sec:bijn}). Adding the vertex $i$ replaces in the product above the subword $[d_m,d_{m+1}]$ by the product $[d_m, i][i, d_{m+1}]$ and this just removes one occurrence of $s_i$. Hence we again have $\ell_{\mathcal{S}}(s_iy)=\ell_{\mathcal{S}}(y)-1$ and the same argument as for the first case gives the conclusion. The proof of the case where $s\in R(w)$ is similar.
\end{proof}
\begin{cor}\label{cor:bi}
Let $w\in\mathcal{W}_f$. The following are equivalent
\begin{enumerate} 
\item One has $s_i\in L(w)\cap R(w)$,
\item One has $s_i<_{\mathcal{T}}\psi(w)~\text{and}~s_i\psi(w)=\psi(w)s_i$,
\item There exists $P\in\Pol(\psi(w))$ which consists of the single edge $(i, i+1)$.
\end{enumerate}
\end{cor}
\begin{proof}
It is an immediate consequence of Proposition \ref{prop:configurations}: $(i,i+1)$ must be a vertex of a polygon with both $i$ initial and $i+1$ terminal (the two other conditions together give a contradiction to the noncrossing property). 
\end{proof}
\begin{cor}\label{cor:bruhat2}
Let $w\in\mathcal{W}_f$. Let $s\in L(w)$, $t\in R(w)$, with $s\neq t$. Then $$s\psi(w)=\psi(w) t~\Leftrightarrow~s=s_j,~t=s_{j-1}~\text{for some index }j.$$
\end{cor}
\begin{proof}
Let $s=s_j$, $t=s_k$ and suppose that $s\psi(w)=\psi(w)t$. Thanks to Proposition \ref{prop:configurations}, applying $s_j$ on the left of $\psi(w)$ either adds or removes the vertex with index $j$. It also possibly adds or removes the vertex with index $j+1$ but in that case, one would have $s\in R(w)$. Since $t\neq s$ the reflection $t$ would then remove a vertex with index $k$ distant from $j$ since any two reflections in $R(w)$ commute with each other, a contradiction to $s\psi(w)=\psi(w)t$ since the operation of $s$ in the left hand side does not change the vertex with index $k$. So we can suppose that $s$ removes or adds the vertex with index $j$, leaving all other vertices of the polygons unchanged. This means that $t$ also has to remove or add the vertex with index $j$. This is possible only if $t=s_{j-1}$ or $t=s_j$ but in the last case we have $s=t$ which is excluded. Conversely, the assumption implies by the above Proposition that $\psi(w)$ has a polygon $P$ having an edge $(j-1, j+1)$. We then have that $s_j\psi(w)=\psi(w) s_{j-1}$ and in the geometrical representation, it corresponds to adding the vertex with index $j$ to the polygon $P$.
\end{proof}
\begin{nota}
Let $w\in\mathcal{W}_f$, $L\subset L(w)$ and $R\subset R(w)$. We build new sets $L'$, $R'$ from $L$ and $R$ by doing the following: if $s\in L\cap R$, we either remove $s$ from $L$ or remove it from $R$. If $s_j\in L$ and $s_{j-1}\in R$, then we either remove $s_j$ from $L$ or remove $s_{j-1}$ from $R$. The process finishes when $L'\cap R'=\emptyset$ and $\{i~|~s_i\in L',~s_{i-1}\in R'\}=\emptyset$. At the end of the process we get two (non canonically defined) sets $L'\subset L$ and $R'\subset R$. It is clear that if $(L', R')$ and $(\tilde{L}', \tilde{R}')$ are two distinct sets with these properties, one has $|L'\cup R'|=|\tilde{L}'\cup\tilde{R}'|$. We call $(L',R')$ and \emp{ending pair} for $(L,R)$.
\end{nota}
\begin{exple}
Let $w=s_2 s_1 s_3$. Then $L(w)=\{s_2\}$, $R(w)=\{s_1, s_3\}$. Let $L=L(w)$, $R=R(w)$. One can choose $L'=\{s_2\}$, $R'=\{s_3\}$. Another possible choice is $L'=\emptyset$, $R'=\{s_1, s_3\}.$
\end{exple}
\noindent We set $$x_{L,R}:=\left(\prod_{s\in L'} s\right)\psi(w)\left(\prod_{s\in R'} s\right)\in \mathcal{W}.$$
Notice that such a notation makes sense only if $x_{L,R}$ is independent of the choice of $(L',R')$. This holds by the following Proposition, which is a generalization of Corollary \ref{cor:bruhat}:
\begin{prop}\label{prop:cube}
Let $w\in\mathcal{W}_f$, $L\subset L(w)$, $R\subset R(w)$. Then $$x_{L,R}:=\left(\prod_{s\in L'} s\right)\psi(w)\left(\prod_{s\in R'} s\right)$$ is independent of the choice of $L'$ and $R'$. Moreover, $x_{L,R}$ lies in $ \mathcal{P}_c$, $x_{L,R}< \psi(w)$
and $\ell_{\mathcal{S}}(x_{L,R})=\ell_{\mathcal{S}}(\psi(w))-|L'\cup R'|$.
\end{prop}
\begin{proof}
We argue by induction on $|L'\cup R'|$. If $|L'\cup R'|=0$ then $L=\emptyset=R$, in which case the claim is trivially true. If $L'\cup R'$ is a singleton, the result follows from Corollaries \ref{cor:bruhat}, \ref{cor:bi} and \ref{cor:bruhat2}. 

Now suppose that $|L'\cup R'|>1$ and let $s_j\in L'\cup R'$. We consider the case $s_j\in L'$, the other case being similar. Write $L''=L'\backslash\{s_j\}$. One can choose $(L'')'=L''$, $(R')'=R'$. Since $s_j\in L(w)$, it follows by Proposition \ref{prop:configurations} that in the representation of $\psi(w)$ by a disjoint union of polygons, we have one of the two following configurations: either $(j, j+1)$ is an edge of a polygon of $\psi(w)$ with $j$ initial, or $j$ is not a vertex of a polygon of $\psi(w)$ but one has a polygon of $\psi(w)$ with an edge $(k, j+1)$ where $k<j$. Now any transposition $s_m\in L''$ is is such that $|m-\ell|>1$ and $R'$ contains neither $s_j$ nor $s_{j-1}$. Using Proposition \ref{prop:configurations} again this implies that any of the two possible configurations are preserved when reducing from $\psi(w)$ to $y:=(\prod_{s\in L''} s)\psi(w)(\prod_{s\in R'} s)$. Indeed, the configuration with an edge $(j, j+1)$ is preserved and since $s_{j-1}\notin R'$ the only thing that can change the edge $(k, j+1)$ of the second configuration is in case we have an edge $(k, j+1)$ with $k<j-1$ and $s_k\in R'$. In that case the edge $(k, j+1)$ is replaced by an edge $(k+1, j+1)$ in $y$ and $y$ still has the second configuration since $k+1<j$. In particular, using the same Proposition, we get $s_j\in L(\varphi(y))$. Induction together with Corollary \ref{cor:bruhat} gives all the claims except that $x_{L,R}$ is independent of the choice of $(L',R')$. Hence let $(\tilde{L}', \tilde{R}')$ be another pair obtained by the process. We must have either $s_j\in \tilde{L}'$, or $s_j\in \tilde{R}'$, or $s_{j-1}\in \tilde{R}'$. Assume for example that $s_j\in\tilde{R}'$, the other cases being similar. Then both $(L'', R')$ and $(\tilde{L}', \tilde{R}'\backslash s_j)$ are both ending pairs for $(L\backslash s_j, R\backslash \{s_j, s_{j-1}\})$, hence by induction since $y$ is independent of the chosen ending pair we have $$y=\left(\prod_{s\in L''} s\right)\psi(w)\left(\prod_{s\in R'} s\right)=\left(\prod_{s\in \tilde{L}'} s\right)\psi(w)\left(\prod_{s\in \tilde{R}'\backslash s_j} s\right).$$

We already know that $s_j\in L(\varphi(y))$ and since we assumed here that $s_j\in \tilde{R}'$ we have in particular $s_j\in R$. Arguing as above using Proposition \ref{prop:configurations} we see that $s_j\in R(\varphi(y))$ in that case. Using Corollary \ref{cor:bi} we get that $$x_{L,R}=\left(\prod_{s\in L'} s\right)\psi(w)\left(\prod_{s\in R'} s\right)=s_jy=ys_j=\left(\prod_{s\in \tilde{L}'} s\right)\psi(w)\left(\prod_{s\in \tilde{R}'} s\right).$$
\end{proof}

\begin{nota}
For $w\in\mathcal{W}_f$, $x\in\mathcal{P}_c$, we set
$$\alpha_w(x):=\ell_{\mathcal{S}}(w)+\ell_{\mathcal{S}}(\psi(w))-\ell_{\mathcal{S}}(x)\in\mathbb{Z},$$
$$\beta_w(x):=\ell_{\mathcal{T}}(x)-\ell_{\mathcal{T}}(\psi(w))\in\mathbb{Z}.$$ 
\end{nota}
\begin{defn}
To each fully commutative element $w\in \mathcal{W}_f$, we will associate an element $X_w$ of $\mathrm{TL}_n$. Set $$Q_w:=\{~x_{L, R}~|~L\subset L(w), R\subset R(w)\}.$$We then define $X_w$ by its coefficients when expressed in Zinno's basis:
$$X_w:=\sum_{x\in Q_w} p_x^w Z_x,$$
where $p_x^w:=(-1)^{\alpha_w(x)} v^{\beta_w(x)}.$
\end{defn}
These elements will turn out to control a part of the base change matrix between $\{Z_x\}_{x\in\Pc}$ and $\{b_w\}_{w\in\Wf}$. 
\begin{rmq}\label{rmq:operation}
As a consequence of Proposition \ref{prop:cube}, one has $s Q_w=Q_w$ for any $s\in L(w)$ and $Q_w s=Q_w$ for any $s\in R(w)$. In particular, $|Q_w|$ is always a power of two and is at least two if $w\neq e$ since for any $w\in\mathcal{W}_f\backslash e$, $L(w)\cup R(w)\neq\emptyset$.
\end{rmq}
\begin{exple}
Let $w=s_1 s_4 s_3 s_2$. Then $\psi(w)=s_1 s_4 s_3 s_2 s_3 s_4$. we have $L(w)=\{s_1,s_4\}$, $R(w)=\{s_2\}$. There is only one possible choice here for the pair $(L',R')$ which is given by $L'=L(w)$, $R'=R(w)$ (the above described process is trivial here). We have (an edge represents a cover relation in the Bruhat order):

$$Q_w=~~~~~~~
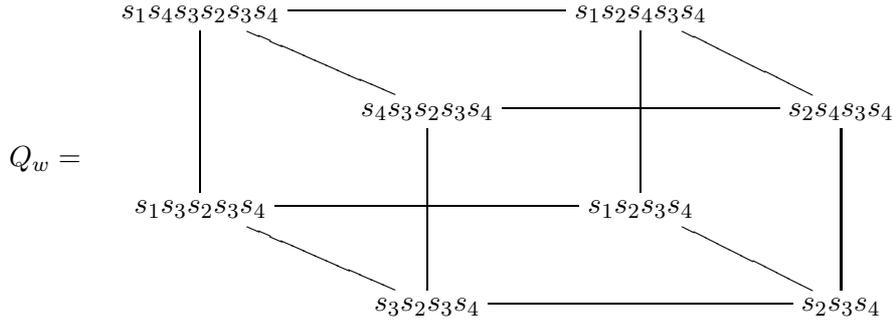
\begin{array}{l}
{\xymatrix{
    s_1 s_4 s_3 s_2 s_3 s_4 \ar@{-}[rr] \ar@{-}[dd] \ar@{-}[dr] && s_1 s_2 s_4 s_3 s_4 \ar@{-}[dr] \ar@{-}[dd]  \\
    &  s_4 s_3 s_2 s_3 s_4 \ar@{-}[rr] \ar@{-}[dd] && s_2 s_4 s_3 s_4 \ar@{-}[dd] \\
    s_1 s_3 s_2 s_3 s_4 \ar@{-}[rr] \ar@{-}[dr] && s_1 s_2 s_3 s_4 \ar@{-}[rd] \\
    & s_3 s_2 s_3 s_4 \ar@{-}[rr] && s_2 s_3 s_4
  }}\end{array}$$
\end{exple}
\begin{exple}\label{exple:xsbs}
Let $s_i\in\mathcal{S}$. We have $L(s_i)=\{s_i\}=R(s_i)$, hence the two possible choices for the pair $(L',R')$ are given by $(\{s_i\}, \emptyset)$ and $(\emptyset, \{s_i\})$. We have $Q_{s_i}=\{e,s_i\}$ and
 $$X_{s_i}=p_{s_i}^{s_i} Z_{s_i}+p_e^{s_i}=-Z_{s_i}+v^{-1}=b_{i}=b_{s_i}.$$
In general for $w\in \Wf$ we have $X_w\neq b_w$.
\end{exple} 
\begin{prop}
The set $\{X_w\}_{w\in\mathcal{W}_f}$ is a basis of the Temperley-Lieb algebra. 
\end{prop}
\begin{proof}
It suffices to order Zinno basis by any linear extension of the order induced by the length function $\ell_{\mathcal{S}}$ restricted to $\mathcal{P}_c$. One then orders the set $\{X_w\}_{w\in\mathcal{W}_f}$ by the order on $\mathcal{W}_f$ obtained as the image of the order we put on $\mathcal{P}_c$ under the bijection $\varphi$. Thanks to Proposition \ref{prop:cube}, one then gets an upper triangular matrix with the invertible coefficients $\{p_{\psi(w)}^w\}_{w\in\mathcal{W}_f}$ on the diagonal, passing from the basis $\{Z_x\}_{x\in \mathcal{P}_c}$ to the set $\{X_w\}_{w\in \mathcal{W}_f}$. Theorem \ref{thm:z} then concludes the proof. 
\end{proof}
\begin{rmq}\label{rmq:invertible}
To achieve triangularity (with invertibility of the diagonal coefficients) of the base change matrix between $\{b_w\}_{w\in\Wf}$ and $\{Z_x\}_{x\in\Pc}$, Zinno orders $\mathcal{P}_c$ by any linear extension of the order induced by the lengths of the lifts $\{\mathbf{m}_x\}_{x\in\mathcal{P}_c}$. Since the braid words $\{\mathbf{m}_x\}_{x\in\mathcal{P}_c}$ are obtained from $m_x$ by replacing each $s_i$ by $\bs_i^{\pm 1}$ and $m_x$ is a reduced word (see Lemma \ref{stdreduced}), it follows that the length of the braid word $\mathbf{m}_x$ is equal to $\ell_{\mathcal{S}}(x)$. Hence Zinno's order is the same order as the one we considered in the proof above. As a consequence, these orders also give triangularity of the base change matrix between $X_w$ and $b_w$, with invertible coefficients on the diagonal. 
\end{rmq}

\begin{lem}\label{lem:xcommeb}
Let $w\in \mathcal{W}_f$.
\begin{enumerate}
\item If $s\in L(w)$, then $b_s X_w=(v+v^{-1}) X_w,$
\item If $s\in L(w)$, then $X_w b_s=(v+v^{-1}) X_w.$
\end{enumerate}
\end{lem}
\begin{proof}
We only prove $(1)$, the proof of $(2)$ being similar. Let $x\in  Q_w$ such that $sx< x$. Since $s\in L(w)$ one has that $sx\in  Q_w\subset\mathcal{P}_c$ by Remark \ref{rmq:operation}. One has either $sx<_{\mathcal{T}} x$, in which case $Z_x=Z_s Z_{sx}$, or $x<_{\mathcal{T}} sx$, in which case $Z_x=Z_s^{-1} Z_{sx}$. Assume that $sx<_{\mathcal{T}} x$. One has $\ell_{\mathcal{T}}(sx)=\ell_{\mathcal{T}}(x)-1$ hence $p_x^w=-vp_{sx}^w$ so we get
\begin{eqnarray*}
b_s(p_{sx}^w Z_{sx}+p_x^w Z_x)&=&(v^{-1}-Z_s)(p_{sx}^w Z_{sx}+p_x^w Z_x)\\
&=&v^{-1} p_{sx}^w Z_{sx}-p_{sx}^w Z_s Z_{sx}+v^{-1} p_x^w Z_x-p_x^w Z_s^2 Z_{sx}\\
&=&v^{-1} p_{sx}^w Z_{sx}+2v^{-1} p_x^w Z_x-p_x^w(Z_s(v^{-1}-v)+1)Z_{sx}\\
&=&(v+v^{-1})(p_{sx}^w Z_{sx}+p_x^w Z_x).
\end{eqnarray*}
Now assume that $x<_{\mathcal{T}} sx$. One has $\ell_{\mathcal{T}}(sx)=\ell_{\mathcal{T}}(x)+1$ hence $p_x^w=-v^{-1} p_{sx}^w$ so we get 
\begin{eqnarray*}
b_s(p_{sx}^w Z_{sx}+p_x^w Z_x)&=&(v-Z_s^{-1})(p_{sx}^w Z_{sx}+p_x^w Z_x)\\
&=&v p_{sx}^w Z_{sx}-p_{sx}^w Z_{x}+v p_x^w Z_x-p_x^w Z_s^{-1} Z_{x}\\
&=&v p_{sx}^w Z_{sx}+2v p_x^w Z_x-p_x^w((v-v^{-1})+Z_s)Z_{x}\\
&=&(v+v^{-1})(p_{sx}^w Z_{sx}+p_x^w Z_x).
\end{eqnarray*}
Summing these equalities on all the couples $(sx, x)$ one gets the result.
\end{proof}

\section{Application: coefficients of the base change matrix between the diagram and Zinno bases}
In this section, we use the newly introduced basis $\{X_w\}_{w\in\mathcal{W}_f}$ to explicitly compute some coefficients of the base change matrix between the diagram and Zinno bases and give a necessary condition for the coefficients to be nonzero. 
\subsection{Zinno basis and Bruhat order}

\begin{lem}\label{lem:wx}
Let $x\in \mathcal{P}_c$. Let $Z_x=\sum_{w\in\mathcal{W}_f} \lambda_w b_w$ be the expansion of $Z_x$ in the diagram basis. If $\lambda_w\neq 0$, then $w< x$.
\end{lem}
\begin{proof}
Recall that $\mathbf{m}_x$ is a braid word whose length is equal to $\ell_{\mathcal{S}}(x)$ representing the simple element $i_c(x)$ in $\Bni$. If one replaces any $\mathbf{s}_i^{\pm}$ by $s_i$ in $\mathbf{m}_x$, one obtains the standard form $m_x$ of $x$, which is an $\mathcal{S}$-reduced expression of $x$. After being mapped to the Temperley-Lieb algebra, any letter $\mathbf{s}_i$ is replaced by $v^{-1}-b_i$ while each letter $\mathbf{s}_i^{-1}$ is replaced by $v-b_i$. Hence if we expand the image of $\mathbf{m}_x$ without reducing, we obtain $2^{\ell_{\mathcal{S}}(x)}$ different terms: for each $\mathbf{s}_i^{\pm 1}$ occuring in $\mathbf{m}_x$ we can either choose the $-b_i$ or the $v^{\pm 1}$. As a consequence, if we expand the image of $\mathbf{m}_x$ in $\mathrm{TL}_n$, we obtain a linear combination of elements of the form $b_{i_1} b_{i_2}\cdots b_{i_k}$ where the corresponding products $s_{i_1} s_{i_2}\cdots s_{i_k}$ are subwords of $m_x$. If such a product $s_{i_1} s_{i_2}\cdots s_{i_k}$ is not an $\mathcal{S}$-reduced expression of a fully commutative element, then $b_{i_1} b_{i_2}\cdots b_{i_k}$ is not a reduced word. Rather $b_{i_1} b_{i_2}\cdots b_{i_k}$ is equal to $(v+v^{-1})^k b_w$ for a unique pair $(k, w)\in\mathbb{Z}_{>0}\times\mathcal{W}_f$ by part $(3)$ of Proposition \ref{prop:jones} and as a consequence of the Temperley-Lieb relations, $w$ has an $\mathcal{S}$-reduced expression which is a subword of $s_{i_1} s_{i_2}\cdots s_{i_k}$. But $s_{i_1} s_{i_2}\cdots s_{i_k}$ was itself a subword of $m_x$. Since $m_x$ is an $\mathcal{S}$-reduced expression of $x$, it follows that $w< x$. 
\end{proof}

As mentioned in Remark \ref{rmq:invertible}, Zinno orders $\mathcal{P}_c$ by any linear extension of the order induced by $\ell_{\mathcal{S}}$. He then proves the following Theorem, which is rewritten here using our notations and Lemma \ref{lem:wx}:
\begin{thm}[{\cite[Theorem 5]{Z}}]\label{thm:zbis}
Let $x\in\mathcal{P}_{c}$, $w\in\Wf$ and assume $w< x$. If $w\neq\varphi(x)$, there exists an element $y\in\mathcal{P}_c$, $y\neq x$ such that $w<y$ and $y<x$.
\end{thm}
\begin{rmq}
In fact, in \cite{Z} the last statement of the Theorem is that the braid length of $m_{\mathbf{y}}$ is smaller that the braid length of $\mathbf{m}_x$, equivalently that $\ell_{\mathcal{S}}(y)<\ell_{\mathcal{S}}(x)$, which is weaker than $y<x$. But in Zinno's proof, various cases are considered to prove the statement and in all them, the element $y$ which is built is a subword of $m_x$. 
\end{rmq}

\noindent Zinno then proves

\begin{prop}[{\cite{Z}}]\label{prop:zin}
Let $x\in\mathcal{P}_{c}$, $w\in\Wf$ and assume $w< x$. If $w\neq\varphi(x)$, then $\ell_{\mathcal{S}}(\psi(w))<\ell_{\mathcal{S}}(x)$.
\end{prop}
\begin{proof}
See \cite{Z}, proof of Theorem $6$.
\end{proof}
\noindent We refine Proposition \ref{prop:zin}:
\begin{prop}\label{prop:refzin}
Let $x\in\mathcal{P}_{c}$, $w\in \Wf$ and assume $w< x$. If $w\neq\varphi(x)$, then $\psi(w)< x$.
\end{prop}
\begin{proof}
Consider the set $$Y_{w,x}:=\{y\in\mathcal{P}_c~|~w< y< x,~y\neq x\}.$$
By Theorem \ref{thm:zbis} we have $Y_{w,x}\neq\emptyset$. Let $y\in Y_{w,x}$ such that $\ls(y)$ is minimal. If $y\neq \psi(w)$, then $\varphi(y)\neq w$. Applying Theorem \ref{thm:zbis} again, we get that there exists $y'\in\mathcal{P}_c$ such that $w< y'< y$, $y\neq y'$, a contradiction to the minimality of $\ls(y)$. Hence $y=\psi(w)$. In other words, $Y_{w,x}$ has a unique element of minimal length which is equal to $\psi(w)$. \end{proof}

Combining Proposition \ref{prop:refzin} with the fact that $w_x=\varphi(x)$ is a subword of the standard form $m_x$ of $x$ we get:
\begin{cor}\label{prop:wsousx}
Let $w\in\mathcal{W}_f$, $x\in\mathcal{P}_c$ and assume that $w< x$. Then $$\psi(w)< x.$$
\end{cor}
\begin{rmq}
The fact that the order hidden behind the Zinno basis is the Bruhat order on $\mathcal{P}_c$ is surprising fact since the natural order is $<_{\mathcal{T}}$, which is the restriction of the refinement order on partitions. In \cite{GobWil}, we give a characterization of the Bruhat order on $\mathcal{P}_c$ and prove that the poset $(\mathcal{P}_c, <)$ is a lattice, isomorphic to the lattice of order ideals in the root poset of type $A_n$, which is also isomorphic to the lattice of Dyck with containment order. \end{rmq}
\subsection{The new basis as an intermediate basis}
We now consider the linear expansion of an element $b_w$ of the diagram basis in terms of the basis $\{X_{w'}\}_{w'\in\Wf}$
$$b_w=\sum_{w'\in\mathcal{W}_f} q_{w'}^{w} X_{w'}$$
and we would like to understand for which $w'$ one can have $q_w^{w'}\neq 0$. To this end, we write the element $X_w$ in the diagram basis as
$$X_w=\sum_{y\in\mathcal{W}_f} r_y^w b_y.$$
\begin{nota}
To each fully commutative element $w\in\mathcal{W}_f$ we associate a subset $F_w\subset\mathcal{W}_f$ defined by
$$F_w=\{y\in\mathcal{W}_f~|~L(y)\supset L(w), R(y)\supset R(w)~\text{and }\psi(y)<\psi(w)\}.$$ 
\begin{rmq}\label{rmq:fw}
Obviously one has $w\in F_w$ and if $y\in F_w$, then $F_y\subset F_w$. The inclusion of these sets defines a new partial order on $\mathcal{W}_f$.
\end{rmq}
\end{nota}
\begin{prop}\label{prop:fw}
If $r_y^w\neq 0$, then $y\in F_w$.
\end{prop}
\begin{proof}
We first show that for any $y\in\Wf$ with $r_y^w\neq 0$ we have $s_i\in L(y)$ for any $s_i\in L(w)$. Thanks to Lemma \ref{lem:xcommeb} one has that 
$$b_{s_i}\underbrace{\left(\sum_{y\in\mathcal{W}_f} r_y^w b_y\right)}_{X_w}=(v+v^{-1})\left(\sum_{y\in\mathcal{W}_f} r_y^w b_y\right).$$
Among all the $y$ for which $r_y^w$ is nonzero, choose an element $y$ such that $\ell_{\mathcal{S}}(y)$ is maximal. It follows from this equality and the maximality of $\ell_{\mathcal{S}}(y)$ that in case $\ell_{\mathcal{S}}(s_iy)>\ell_{\mathcal{S}}(y)$, then $s_iy$ cannot be a fully commutative element. In other words, when reducing $b_{s_i} b_y$, one has to apply the relation $b_{s_i}^2=(v+v^{-1})b_{s_i}$ (in case $s_iy< y$) or the relation $b_{s_i} b_{s_{i\pm 1}} b_{s_i}=b_{s_i}$ (in case $s_iy>_{\mathcal{S}} y$). In the first case $y$ has an $\mathcal{S}$-reduced expression beginning with $s_i$, hence $s_i\in L(y)$ implying $b_{s_i} b_y=(v+v^{-1}) b_y$. In the second case $y$  has an $\mathcal{S}$-reduced expression beginning with $s_{i\pm 1}s_i$. But $b_y$ also appears in the right hand side and an element beginning by $b_{s_{i\pm 1}} b_{s_i}$ obviously cannot come from an element $b_{s_i} b_{y''}$ with $y''\in\mathcal{W}_f$. This gives a contradiction. Hence it follows that our element $y$ has an $\mathcal{S}$-reduced expression beginning with $s_i$, that it, $s_i$ lies in $L(y)$. We can then remove $b_{s_i} b_y=(v+v^{-1}) b_y$ from both sides of the equality above obtaining  
$$b_{s_i}\left(\sum_{z\in\mathcal{W}_f,z\neq y} r_z^w b_z\right)=(v+v^{-1})\left(\sum_{z\in\mathcal{W}_f, z\neq y} r_z^w b_z\right).$$
One can then choose another element $z$ with maximal length among the remaining ones with nonzero coefficient and apply the same argument to show that $s_i\in L(z)$ and so on until we run out of all the elements with nonzero coefficient. This proves that for any $s_i\in L(w)$, we have $s_i\in L(y)$ for any $y$ such that $r_y^w\neq 0$. Doing the same for any $s_j\in R(w)$ one gets that for any $y$ such that $r_y^w\neq 0$, $L(y)\supset L(w)$ and $R(y)\supset R(w)$. 

Now if $y$ is such that $r_y^w\neq 0$, one must have $y<x$ for at least one $x\in  Q_w$ by Lemma \ref{lem:wx}. Thanks to Proposition \ref{prop:wsousx}, we have $\psi(y)< x$ and thanks to Proposition \ref{prop:cube} one also has that $x< \psi(w)$ giving $\psi(y)< \psi(w)$. Therefore we have that $y\in F_w$.
\end{proof}
\begin{prop}\label{prop:coeffq}
If $q_{w'}^w\neq 0$, then $w'\in F_w$.
\end{prop}
\begin{proof}
We proceed by induction of $\ell_{\mathcal{S}}(\psi(w))$. If $\ell_{\mathcal{S}}(\psi(w))=1$ then $w$ is a simple transposition. In that case by Example \ref{exple:xsbs} one has $b_w=X_w$ and the claim is trivially true since $F_w=\{w\}$. Now suppose that $\ell_{\mathcal{S}}(\psi(w))>1$. Thanks to the previous Proposition we have that 
$$X_w=\sum_{y\in F_w} r_y^w b_y,$$
in particular, $\psi(y)< \psi(w)$, hence $\ell_{\mathcal{S}}(\psi(y))<\ell_{\mathcal{S}}(\psi(w))$ in case $w\neq y$, $y\in F_w$. Hence by induction one has that 
$$b_y=\sum_{z\in F_y} q_z^y X_z$$ which we replace in the previous equality:
$$X_w=r_w^w b_w+\sum_{y\in F_w,y\neq w} r_y^w \left(\sum_{z\in F_y} q_z^y X_z\right).$$
But since $y\in F_w$, one has that $F_y \subset F_w$ (see Remark \ref{rmq:fw}), hence the equality can be rewritten as
$$X_w=r_w^w b_w+\sum_{y\in F_w, y\neq w} \widetilde{q}_y^w X_y$$
for suitable polynomials $\widetilde{q}_y^w$. This concludes the proof since $r_w^w$ is invertible by Remark \ref{rmq:invertible}.
\end{proof}
\noindent Now write the expansion of an element $b_w$ in Zinno basis as
$$b_w=\sum_{x\in \mathcal{P}_c} h_x^w Z_x.$$
As an immediate consequence of Proposition \ref{prop:coeffq} we get:
\begin{cor}\label{cor:cdn}
If $x\notin\bigcup_{y\in F_w} Q_y$, then $h_x^w=0$.
\end{cor}
\begin{lem}\label{lem:longueurs}
Let $w\in\Wf$, $x:=\psi(w)\in\Pc$. Then $$\frac{\ell_{\mathcal{S}}(x)-\ell_{\mathcal{T}}(x)}{2}=\ell_{\mathcal{S}}(w)-\ell_{\mathcal{T}}(x).$$
\end{lem}
\begin{proof}
Recall that there is a rule to read $w_{x}$ which is a reduced expression for $w$ as a subword of $m_x$ that we recalled in Example \ref{ex:zinn} and in the paragraphs above it. Notice that $\ell_{\mathcal{T}}(x)$ is the number of syllables (or centers) of $m_x$. The claimed equality holds since all the centers contribute to $w_x$. The left hand side of the above equality is equal to the number of all the contributions to $w$ different from the centers. Indeed, recall from Section \ref{sub:zinnob} that any syllable contributes any of its simple transpositions exactly once to $w$ and that if $s_i$ is not at the center, it occurs twice in the syllable.
\end{proof}

\begin{lem}[Zinno, \cite{Z}]\label{lem:coeffdiag}
Let $w\in\mathcal{W}_f$ and $x:=\psi(w)$. The coefficient of $b_w$ in the expansion of $Z_{x}$ in the diagram basis is equal to $$(-1)^{\ell_{\mathcal{S}}(w)} v^{-2 k_w+\ell_{\mathcal{S}}(w)-\ell_{\mathcal{T}}(x)},$$
where $k_w$ is the number of letters of $\mathbf{m}_x$ which have negative exponent and contribute to $w_x$.
\end{lem}
\begin{proof}
The coefficient on the diagonal is explicitly computed by Zinno in \cite{Z} at the end of Section $6$. Since we have different notation and conventions we sketch a proof. 

Let $x=\psi(w)$. Recall that $Z_x$ is the image in $\mathrm{TL}_n$ of the element of $\Bni$ represented by the braid word $\mathbf{m}_x$. It is obtained by replacing each letter $\mathbf{s}_i$ in $\mathbf{m}_x$ by $v^{-1}-b_i$ and each letter $\mathbf{s}_i^{-1}$ by $v-b_i$. Hence if we expand without reducing, we obtain $2^{\ell_{\mathcal{S}}(x)}$ different terms: for each $\mathbf{s}_i^{\pm 1}$ occuring in $\mathbf{m}_x$ we can either choose the $-b_i$ or the $v^{\pm 1}$. Zinno proves at the beginning of Section $6$ of \cite{Z} that among the $2^{\ell_{\mathcal{S}}(x)}$ terms which are (possibly non reduced) words in the $b_i$ multiplied by a power of $v$, the term obtained by taking the $b_i$ from any $\mathbf{s}_i^{\pm 1}$ contributing to $w_x$ and taking the $v^{\pm 1}$ from any other $\mathbf{s}_i^{\pm 1}$ is the only term among the $2^{\ell_{\mathcal{S}}(x)}$ which is proportional to $b_w$. But its coefficient is easily computed. Indeed, each $b_i$ which is contributed is multiplied by $-1$, and since a $b_i$ is contributed exactly from the $\mathbf{s}_i^{\pm 1}$ in $\mathbf{m}_x$ contributing to $w_x$ and since moreover $w_x$ is an $\mathcal{S}$-reduced expression of $w$, this gives rise to a sign $(-1)^{\ell_{\mathcal{S}}(w)}$. Now each $\mathbf{s}_i^{\pm 1}$ not contributing to $w_x$ must contribute its $v^{\pm 1}$. For any $\mathbf{s}_i^{-1}$ contributing to $w_x$, there is an $\mathbf{s}_i\mapsto v^{-1}-b_i$ at its right which does not contribute, giving a coefficient $v^{-k_w}$. Now if a $\mathbf{s}_i$ contributes to $w_x$, it means that $\mathbf{s}_i$ is the center of a syllable. As a consequence all the $\mathbf{s}_i^{-1}$ do not contribute to $w_x$. We need to count them. The number of occurrences of all the various $\mathbf{s}_i^{\pm 1}$ with $s_i$ occuring at a center is given by $\ell_{\mathcal{S}}(x)-2 k_w$. We then need to subtract the centers and there are $\ell_{\mathcal{T}}(x)$ many of them. We then need to divide the result by two since we have here all the $\mathbf{s}_i^{\pm 1}$ such that the instance with positive exponent contribute with the centers removed, but any instance $\mathbf{s}_i^{\pm 1}$ of one of these comes with an instance of $\mathbf{s}_i^{\mp 1}$ in the same syllable since we removed the centers. Hence the power of $v$ we obtain from the $\mathbf{s}_i^{-1}$ not contributing to $w_x$ is equal to 
$$\frac{\ell_{\mathcal{S}}(x)-2 k_w-\ell_{\mathcal{T}}(x)}{2}$$
so the power of $v$ multiplying $b_w$ in the expansion equals
$$-k_w+\frac{\ell_{\mathcal{S}}(x)-2 k_w-\ell_{\mathcal{T}}(x)}{2}=-2 k_w+\ell_{\mathcal{S}}(w)-\ell_{\mathcal{T}}(x),$$
where the last equality follows from Lemma \ref{lem:longueurs}.
\end{proof}
\begin{prop}\label{thm:coeff}
Let $w\in\mathcal{W}_f$, $x\in Q_w$ and $b_w=\sum_{x\in\Pc} h_x^w Z_x$. Then $$h_x^w=(-1)^{\alpha_w(x)}v^{2 k_w+\ell_{\mathcal{T}}(x)-\ell_{\mathcal{S}}(w)},$$
where $k_w$ is the number of letters of $\mathbf{m}_{{\psi(w)}}$ which have negative exponent and contribute to $w_{{\psi(w)}}$.
\end{prop}
\begin{proof}
This is a consequence of the fact that if $y\in F_w, y\neq w$, then $Q_w\cap Q_y=\emptyset$. Indeed, assume that $x\in Q_w\cap Q_y$. Then there exist two sets $L'\subset L(w)$, $R'\subset R(w)$ such that
$$\left(\prod_{s\in L'} s\right)x\left(\prod_{s\in R'} s\right)=\psi(w).$$
Since $L(y)\supset L(w)$ and $R(y)\supset R(w)$ and $x\in Q_y$, one also has using Remark \ref{rmq:operation} that 
$$\psi(w)=\left(\prod_{s\in L'} s\right)x\left(\prod_{s\in R'} s\right)\in Q_y.$$
But by definition of $F_w$, $\psi(y)<\psi(w)$, $\psi(w)\neq\psi(y)$. Since any element $z\in Q_y$ satisfies $z<\psi(y)$ it follows that $\psi(w)<\psi(y)<\psi(w)$, a contradiction.

As a consequence of this observation together with Corollary \ref{cor:cdn}, if one knows the coefficient of $Z_{\psi(w)}$ in the expansion of $b_w$, one knows the coefficient of any $Z_x$ for $x\in Q_w$ since the only element of the basis $\{X_w\}$ which can contribute elements $Z_x$ for $x\in Q_w$ is $X_w$. Using Lemma \ref{lem:coeffdiag} we have that the inverse coefficient of $b_w$ in the expansion of $Z_{\psi(w)}$ is equal to $(-1)^{\ell_{\mathcal{S}}(w)} v^{-2 k_w+\ell_{\mathcal{S}}(w)-\ell_{\mathcal{T}}(\psi(w))}$. Therefore since the base change matrix is upper triangular with invertible coefficient on the diagonal one has that the coefficient of $Z_{\psi(w)}$ in the expansion of $b_w$ is given by $$(-1)^{\ell_{\mathcal{S}}(w)} v^{2 k_w-\ell_{\mathcal{S}}(w)+\ell_{\mathcal{T}}(\psi(w))}.$$ Using the fact that 
$$b_w=\sum_{w'\in F_w} q_{w'}^w X_{w'}$$
and that any element $Z_x$ with $x\in Q_w$ is contributed exclusively by $X_w$, one has that
$$q_w^w p_{\psi(w)}^w=(-1)^{\ell_{\mathcal{S}}(w)} v^{2 k_w-\ell_{\mathcal{S}}(w)+\ell_{\mathcal{T}}(\psi(w))},$$
hence $q_w^w=v^{2 k_w-\ell_{\mathcal{S}}(w)+\ell_{\mathcal{T}}(\psi(w))}$ since $p_{\psi(w)}^w=(-1)^{\ell_{\mathcal{S}}(w)}$. Hence for any $x\in Q_w$ we obtain $$h_x^w=q_w^w p_x^w=(-1)^{\alpha_w(x)}v^{2 k_w+\ell_{\mathcal{T}}(x)-\ell_{\mathcal{S}}(w)},$$
as claimed.
\end{proof}
\noindent Putting Corollary \ref{cor:cdn} and Proposition \ref{thm:coeff} together we have:
\begin{thm}\label{thm:final}
Let $w\in\mathcal{W}_f$, $x\in\mathcal{P}_c$.
\begin{enumerate}
\item If $x\notin\bigcup_{y\in F_w} Q_y$, then $h_x^w=0$.
\item If $x\in Q_w$, then $h_x^w=(-1)^{\alpha_w(x)}v^{2 k_w+\ell_{\mathcal{T}}(x)-\ell_{\mathcal{S}}(w)}$.
\end{enumerate}
\end{thm}
\begin{rmq}
Theorem \ref{thm:final} explicitly gives the coefficient $h_x^w$ of the base change matrix between the Zinno and diagram bases except in case $x\in\bigcup_{y\in F_w,y\neq w} Q_y$. There are two main difficulties in that case: firstly, examples show that the condition of \ref{cor:cdn} is not sufficient, and secondly, one may have $y,y'\in F_w\backslash w$ such that $Q_y\cap Q_{y'}\neq\emptyset$, hence different sets $Q_y$, $Q_{y'}$ may contribute the same element $Z_x$ of the Zinno basis for $x\in Q_y\cap Q_{y'}$; the sum of the various contributions may be zero, but in some cases it is nonzero, giving rise to a coefficient $h_x^w$ which is not a monomial. It is the case for example, as computations with GAP \cite{gap} show, in type $A_4$ for $x=s_4s_3s_2s_1s_2s_3s_4=(1,5)$ and $w=c^{-1}=s_4s_3s_2s_1$ where $h_x^w$ is not a monomial.
\end{rmq}

\end{document}